\documentclass[
11pt,
a4paper,
reqno,
]{amsart}

\usepackage{a4wide}
\usepackage{amsfonts}
\usepackage{amsmath,latexsym,amssymb,amsfonts,amsbsy, amsthm}
\usepackage[
allcolors=blue,
colorlinks=true,
pdfpagelabels,
pagebackref]{hyperref}

\usepackage{mathtools}

\usepackage[usenames]{color}
\usepackage{xspace,colortbl}
\usepackage[utf8]{inputenc}
\allowdisplaybreaks
\usepackage[T1]{fontenc}
\usepackage{enumitem}

\usepackage{lmodern}
\usepackage{microtype}

\usepackage{bbm}

\theoremstyle{plain}
\newtheorem{theorem}{Theorem}[section]
\newtheorem{definition}[theorem]{Definition}
\newtheorem{corollary}[theorem]{Corollary}
\newtheorem{proposition}[theorem]{Proposition}
\newtheorem{lemma}[theorem]{Lemma}
\newtheorem{remark}[theorem]{Remark}

\numberwithin{theorem}{section} \numberwithin{equation}{section}

\def\R{\mathbb{R}}

\newcommand{\del}{\partial }
\renewcommand{\phi}{\varphi}
\newcommand{\N}{\mathbb{N}}

\newcommand{\cA}{{\mathcal A}}

\newcommand{\cI}{{\mathcal I}}

\newcommand{\cN}{{\mathcal N}}

\newcommand{\weakto}{\rightharpoonup}

\newcommand{\B}{{\bf B}}

\newcommand{\eps}{\varepsilon}
\renewcommand{\epsilon}{\varepsilon}

\DeclareOldFontCommand{\bf}{\normalfont\bfseries}{\mathbf}
\begin{document}
	\title[
	Spectrum of a mixed-type operator and rotating wave solutions % of NLKG
	]{On the spectrum of a mixed-type operator with applications to rotating 
	wave solutions}
	%
	%
%
%%%%%%%%%%%%%%%%%%%%%%%%%%%%%%%%%%%%%%%%%%%%%%%%%%%%%%%%%%%%%%%%%%%%%%%%%%%%%%%%%%%%%%%%%%%%%%%%%%%%%%%%%%%%
	%
	\author{Joel K\"ubler} 
	\address{Institut f\"ur Mathematik,
	Goethe-Universit\"at Frankfurt,
	Robert-Mayer-Str. 10,
	D-60629 Frankfurt am Main, Germany}
   \email{kuebler@math.uni-frankfurt.de}
%
%%%%%%%%%%%%%%%%%%%%%%%%%%%%%%%%%%%%%%%%%%%%%%%%%%%%%%%%%%%%%%%%%%%%%%%%%%%%%%%%%%%%%%%%%%%%%%%%%%%%%%%%%%%%
    %
	\date{\today}

	%\subjclass[2010]{...} 
%
%%%%%%%%%%%%%%%%%%%%%%%%%%%%%%%%%%%%%%%%%%%%%%%%%%%%%%%%%%%%%%%%%%%%%%%%%%%%%%%%%%%%%%%%%%%%%%%%%%%%%%%%%%%%
%
	\begin{abstract}
		We study rotating wave solutions of the nonlinear wave equation 
		$$
			\left\{ 
			\begin{aligned}
				\partial_{t}^2 v - \Delta v + m v & = |v|^{p-2} v   \quad && 
				\text{in $\mathbb{R} 
					\times \mathbf{B}$} \\
				v & = 0 && \text{on $\mathbb{R} \times \partial \mathbf{B}$}
			\end{aligned}
			\right.
		$$
		where $2<p<\infty$, $m \in \R$ and $\mathbf{B} \subset \mathbb{R}^2$ 
		denotes the unit 
		disk. If the angular velocity $\alpha$ of the rotation is larger than 
		$1$, this leads 
		to a semilinear boundary value problem on $\B$ involving a mixed-type 
		operator, whose spectrum is related to the zeros of Bessel functions  
		and could 
		generally be badly behaved. 
		Based on new estimates for these zeros, we find 
		values of $\alpha$ such that the spectrum only consists of eigenvalues 
		with finite multiplicity and has no accumulation point. Combined with 
		suitable spectral 
		estimates, this allows us to formulate an appropriate indefinite 
		variational 
		setting 
		and find ground state solutions of the reduced equation for $p \in 
		(2,4)$. Using a 
		minimax characterization of the ground state 
		energy, we ultimately show that these ground states are nonradial and 
		thus yield nontrivial rotating waves, provided $m$ is sufficiently 
		large.
  \end{abstract} 
	\maketitle
\section{Introduction}
We consider time-periodic solutions of the nonlinear wave equation 
\begin{equation} \label{NLKG}
	\left\{ 
	\begin{aligned}
		\del_{t}^2 v - \Delta v + m v & = |v|^{p-2} v   \quad && \text{in $\R 
		\times \B$} \\
		v & = 0 && \text{on $\R \times \del \B$}
	\end{aligned}
	\right.
\end{equation}
where $2<p<\infty$, $m \in \R$ and $\B \subset \R^2$ denotes the unit disk. In 
the case $m>0$, this is also commonly referred to as a nonlinear Klein-Gordon 
equation.
A well-known class of such solutions is given by standing wave solutions, which 
reduce \eqref{NLKG} either to a 
stationary nonlinear Schr\"odinger or a nonlinear 
Helmholtz equation and 
have been studied extensively on the whole space $\R^N$, see 
\cite{Strauss,evequoz-weth}. 
Note that this yields complex-valued solutions whose amplitude remains 
stationary, however, while other types of time-periodic solutions are 
significantly less well understood. 
In particular, much less is known about the dynamics of nonlinear wave 
equations in general 
bounded domains. 

In the one-dimensional setting, which typically describes the forced 
vibrations of a nonhomogeneous string, the existence of time-periodic 
solutions satisfying either Dirichlet or periodic boundary conditions has been  
treated in the seminal works of Rabinowitz~\cite{Rabinowitz} 
and Br\'ezis, Coron and Nirenberg~\cite{Brezis-Coron-Nirenberg} by variational 
methods, but the results in higher dimensions are more sparse. On
balls centered at the origin, the existence of radially symmetric time-periodic
solutions was first studied by Ben-Naoum and 
Mawhin~\cite{Ben-Naoum-Mawhin} for sublinear 
nonlinearities and subsequently received further attention, see e.g. the recent 
works of
Chen and Zhang~\cite{Chen-Zhang,Chen-Zhang-2,Chen-Zhang-3} and the references 
therein.

In this paper, we study \emph{rotating wave solutions} as introduced in 
\cite{Kuebler-Weth}, which are time-periodic real-valued solutions of 
\eqref{NLKG} given by the ansatz 
\begin{equation} \label{eq: Rotating solution ansatz}
	v(t,x)=u(R_{\alpha t} (x)),
\end{equation}
where $R_\theta \in O(2)$ describes a rotation in $\R^2$ with angle $\theta>0$, 
i.e.,
\begin{equation} \label{eq: Rotation definition} 
R_{\theta}(x)= (x_{1} \cos \theta  + x_2 \sin \theta , -x_{1} \sin \theta  + 
x_2 \cos \theta ) \qquad \text{for $x \in \R^2$}.
\end{equation}
In particular, the constant $\alpha>0$ in \eqref{eq: Rotating solution ansatz} 
is the angular velocity of the rotation.
Consequently, such solutions can be interpreted as rotating waves in a 
nonlinear medium. 
We note that a related ansatz for generalized traveling waves on manifolds has 
also been considered in \cite{Taylor,Mukherjee,Mukherjee2}, while a class of 
spiral shaped solutions for a nonlinear Schr\"odinger equation on $\R^3$ has 
been treated in \cite{Agudelo-Kuebler-Weth}.

In the following, we let $\theta$ denote the angular variable in 
two-dimensional 
polar coordinates and note that the ansatz \eqref{eq: Rotating solution ansatz} 
reduces 
\eqref{NLKG} to
\begin{equation} \label{Reduced equation}
	\left\{ 
	\begin{aligned} 
		-\Delta u + \alpha^2 \del_{\theta}^2 u +m u & = |u|^{p-2} u \quad 
		&& \text{in $\B$} \\
		u & = 0 && \text{on $\del \B$}
	\end{aligned}
	\right.
\end{equation}
where $\del_\theta = x_{1} \del_{x_2} - x_{2} \del_{x_{1}}$ then corresponds to 
the angular derivative.
Note that this equation has solutions which are independent of $\theta$, but  
these correspond to stationary and therefore non-rotating solutions of 
\eqref{NLKG}. 
In the following, our goal is to prove the existence of nonradial, i.e., 
$\theta$-dependent, solutions of 
\eqref{Reduced equation}.

In the case $\alpha \leq 1$, this question has been studied in great detail in
\cite{Kuebler-Weth}, where a connection to degenerate Sobolev inequalities is explored. In particular, it has been observed that the 
ground states, i.e., minimizers of the associated Rayleigh quotient, are nonradial in certain parameter regimes for $p$ and $\alpha$. 

The main purpose of the present paper is the study of nonradial solutions of 
\eqref{Reduced equation} for $\alpha>1$.
However, the direct variational methods employed in \cite{Kuebler-Weth} cannot 
be extended to this case since the operator
$$
L_\alpha \coloneqq -\Delta + \alpha^2 \del_{\theta}^2
$$
is neither elliptic nor degenerate elliptic, and the associated Rayleigh quotient becomes unbounded 
from below, see \cite[Remark 4.3]{Kuebler-Weth}. Indeed, note that in 
polar coordinates $(r,\theta) \in (0,1) \times 
(-\pi,\pi)$ we have
$$
L_\alpha u = - \del_{r}^2 u - \frac{1}{r} \del_r u - \left( \frac{1}{r^2} - 
\alpha^2 \right) \del_{\theta}^2 u
$$
and hence the operator is in fact of mixed-type for $\alpha>1$:
It is elliptic in the smaller ball $B_{1/\alpha}(0)$ of radius ${1}/{\alpha}$, 
parabolic on the 
sphere of radius ${1}/{\alpha}$ and hyperbolic in the annulus $\B \setminus \overline{B_{1/\alpha}(0)}$. 
%The operator is further identified to be of Tricomi type. 
In general, such operators are difficult to deal with via variational methods, 
and 
instead results often rely on 
separate treatments of the different regions of specific type and then gluing 
the solutions together, see e.g. \cite{Morawetz,Otway} for more details.

From a functional analytic viewpoint, the quadratic form associated to 
$L_\alpha$ is strongly 
indefinite, i.e., it is negative on an infinite-dimensional subspace. 
Classically, related problems have been treated for operators of the form 
$-\Delta - E$ on $\R^N$ where $E \in \R$ lies in a spectral gap of the 
Laplacian. In this direction, we mention the use of a dual variational 
framework in order to prove the existence of nonzero solutions of a nonlinear 
stationary Schr\"odinger equation in \cite{Alama-Li}, as well as
abstract operator theoretic methods used in \cite{Buffoni-Jeanjean-Stuart} for 
a related problem. However, both of these exemplary 
approaches require specific assumptions regarding spectral 
properties of the associated operator. Moreover, the sole existence of nonzero 
solutions to 
\eqref{Reduced equation} is insufficient in our case since we are interested in nontrivial \emph{rotating} wave solutions.

In the present case of problem~(\ref{Reduced equation}), a main obstruction, in addition to the unboundedness of the spectrum of the linear operator $L_\alpha$ from above and below, is the possible existence of finite accumulation points of this spectrum.
As a first step, we therefore analyze the spectrum of $L_\alpha$ in detail, which is closely 
related to the spectrum of the Laplacian and thus the zeros of Bessel 
functions. In fact, the Dirichlet eigenvalues of $L_\alpha$ are given by
$$
j_{\ell,k}^2-\alpha^2 \ell^2 ,
$$
where $\ell \in \N_0$, $k \in \N$ and $j_{\ell,k}$ denotes the $k$-th zero of 
the Bessel function of the first kind $J_{\ell}$. In particular, 
the 
structure of the spectrum heavily depends on the asymptotic behavior of 
the zeros of these Bessel functions. Despite this explicit characterization, it 
is not 
clear whether the spectrum of $L_\alpha$ only consists of isolated points. 
Indeed, known results on the asymptotics of the zeros of Bessel functions turn 
out to be insufficient to exclude accumulation points or even density in $\R$.
In fact, similar spectral issues arise in the study of radially symmetric 
time-periodic solutions of \eqref{NLKG} on balls  $B_a(0)$, where the 
spectral properties of the radial wave operator are intimately connected to the 
arithmetic properties of the ratio between the radius $a>0$ and the period 
length, see  e.g.
\cite{Mawhin-Survey, Berkovits-Mawhin} and the references therein for more 
details.

This turns out to be a serious obstruction for the use of variational methods 
and thus necessitates a detailed analysis of the asymptotic behavior of 
different sequences of zeros. Our first main result then characterizes the 
spectrum of 
$L_\alpha$ 
as follows.
\begin{theorem}  \label{Theorem: Spectrum Characterization - Introduction}
	For any $\alpha>1$ the spectrum of $L_\alpha$ is unbounded from above and 
	below. Moreover, there exists an unbounded sequence $(\alpha_n)_n  \subset 
	(1,\infty)$ such that the following properties hold for $n \in \N$:
        \begin{itemize}
        \item[(i)] The spectrum of $L_{\alpha_n}$ consists of eigenvalues with 
	finite multiplicity.
      \item[(ii)] There exists $c_n>0$ 
	such that for each 
	$\ell \in \N_0$, $k \in \N$ we either have $j_{\ell,k}^2-\alpha_n^2 
	\ell^2=0$ or
	\begin{equation} \label{eq: eigenvalue lower bound} 
	|j_{\ell,k}^2-\alpha_n^2 \ell^2| \geq c_n j_{\ell,k}.
	\end{equation} 
      \item[(iii)] The spectrum of $L_{\alpha_n}$ has no finite 
	accumulation points.
       \end{itemize}
\end{theorem}
The proof of this result is based on the observation that the formula
$$
j_{\ell,k}^2-\alpha^2 \ell^2=(j_{\ell,k}+\alpha \ell) \ell 
\left(\frac{j_{\ell,k}}{\ell}-\alpha\right)
$$  
implies that for any unbounded sequences $(\ell_i)_i$, $(k_i)_i$, the 
corresponding 
sequence of eigenvalues $j_{\ell_i,k_i}^2-\alpha^2 \ell_i^2$ can only remain 
bounded if 
\begin{equation} \label{eq:Intro:Eigenvalue condition}
\frac{j_{\ell_i,k_i}}{\ell_i}-\alpha \to 0
\end{equation}
as $i \to \infty$. It turns out that \eqref{eq:Intro:Eigenvalue condition} can 
only hold if ${\ell_i}/{k_i} \to \sigma$, where $\sigma = \sigma(\alpha)>0$ 
is uniquely determined and can be characterized via a transcendental equation. 
This motivates a more detailed investigation of $j_{\sigma 
k,k}$, $k \in \N$ which gives rise to a new estimate for $j_{\ell,k}$, $\ell 
\in \N_0$, $k \in \N$, see Lemma~\ref{Lemma: iota bounds} and 
Remark~\ref{Remark: 
iota 
bounds estimate} below. In order to estimate arbitrary sequences in 
\eqref{eq:Intro:Eigenvalue condition}, we 
are then forced to restrict the problem to velocities $\alpha = 
\alpha_n$ such that the associated values 
$\sigma_n=\sigma(\alpha_n)$ are suitable rational numbers.
The fact that such a restriction is necessary is not surprising when compared 
to similar properties observed for the radial wave operator as mentioned above.

Theorem~\ref{Theorem: Spectrum Characterization - Introduction} then plays a 
central role in the formulation of a variational framework for \eqref{Reduced 
equation} and allows us to recover sufficient regularity properties for 
$L_{\alpha_n}$. More specifically, for $\alpha=\alpha_n$ we may then define a 
suitable Hilbert space $E_{\alpha,m}$ whose norm is related to the 
quadratic form 
$$
u \mapsto  \int_\B \left( |\nabla  u|^2 - \alpha^2 |\del_\theta 
u|^2 +m u^2 \right)\, dx,
$$
see Section~\ref{Section: Variational Formulation} 
below 
for details. 
The space $E_{\alpha,m}$ admits a decomposition of the form
$$
E_{\alpha,m} = E_{\alpha,m}^+ \oplus F_{\alpha,m} ,
$$
where the spaces $E_{\alpha,m}^+$ and $ F_{\alpha,m}$ essentially correspond to 
the 
eigenspaces of positive and nonpositive eigenvalues of $-\Delta + \alpha^2 
\del_\theta^2 +m$, respectively. Crucially, the estimate \eqref{eq: eigenvalue 
lower bound} and 
fractional Sobolev embeddings allow us to deduce that $E_{\alpha,m}$ compactly 
embeds 
into $L^p(\B)$ for $p \in (2,4)$.

We may then find solutions of \eqref{Reduced equation} as critical points of 
the associated energy functional $\Phi_{\alpha,m}: E_{\alpha,m} \to \R$ given by
$$
\Phi_{\alpha,m}(u)\coloneqq 
\frac{1}{2} \int_\B \left( 
|\nabla  
u|^2 - 
\alpha^2 |\del_\theta u|^2 + m u^2 \right) \, dx - \frac{1}{p} \int_\B |u|^p \, 
dx . 
$$
Due to the strongly indefinite nature of \eqref{Reduced equation}, 
$\Phi_{\alpha,m}$ is 
unbounded from above and below and does not possess a mountain pass structure 
so, in particular, the classical mountain pass theorem and its variants are not 
applicable. 
Instead, we consider the generalized Nehari manifold introduced by 
Pankov~\cite{Pankov}
$$
\cN_{\alpha,m}\coloneqq \left\{ u \in E_{\alpha,m} \setminus F_{\alpha,m}: \ 
\Phi_{\alpha,m}'(u) u=0 \ \text{and } \Phi_{\alpha,m}'(u)v = 0 \ \text{for all 
$v \in F_{\alpha,m}$} \right\} .
$$
Using further abstract results due to Szulkin and Weth~\cite{Szulkin-Weth}, we 
can then 
show that
$$
c_{\alpha,m}=\inf_{u \in \cN_{\alpha,m}} \Phi_{\alpha,m}(u)
$$
is positive and attained by a critical point of $\Phi_{\alpha,m}$ for 
$\alpha=\alpha_n$ as 
in Theorem~\ref{Theorem: Spectrum Characterization - Introduction} 
and $m \in \R$. In particular, such a 
minimizer then necessarily has minimal energy among all critical points of 
$\Phi_{\alpha,m}$, and is therefore referred to as a \emph{ground state 
solution} or \emph{ground state} of 
\eqref{Reduced equation}.

In general, it is not clear whether such a ground state is nonradial.
Our second main result further states that \eqref{Reduced equation} has 
nonradial ground state solutions for certain choices of parameters.
\begin{theorem} \label{Theorem: Existence of nonradial solutions - Introduction}
	Let $p \in (2,4)$ and let the sequence $(\alpha_n)_n  \subset 
	(1,\infty)$ be given by Theorem~\ref{Theorem: Spectrum Characterization - 
	Introduction}. Then the following properties hold:
	\begin{itemize}
		\item[(i)]
		For any $n \in \N$ and $m \in \R$ there exists a ground state solution 
		of \eqref{Reduced equation} for $\alpha=\alpha_n$.
		\item[(ii)]
		For any $n \in \N$ there exists $m_n>0$ 
		such that the ground state solutions of \eqref{Reduced equation} are 
		nonradial 
		for 
		$\alpha = \alpha_n$ and $m >m_n$. 
	\end{itemize}
\end{theorem}
In fact, we can prove a slightly more general result in the sense that the 
statement of Theorem~\ref{Theorem: Existence of nonradial solutions - 
Introduction} holds 
whenever the kernel of $L_{\alpha}$ is finite-dimensional 
and an inequality of the form \eqref{eq: eigenvalue lower bound} holds. 
The proof is essentially based on an energy comparison, noting that the minimal 
energy of the unique positive radial solution can be estimated from below in 
terms of $m$. 
Using a minimax characterization of $c_{\alpha,m}$, we can then show that this 
ground state energy grows slower than the radial energy as $m \to \infty$. 

Throughout the paper, we only consider real-valued solutions and consequently 
let all function spaces be real. Nonradial complex-valued solutions of 
\eqref{Reduced equation}, on the 
other hand, 
can be found much more easily using constrained minimization over suitable 
eigenspaces. This technique has been applied to a related problem in 
\cite{Taylor}. We point out, however, that the modulus of such solutions is 
necessarily 
radial, while Theorem \ref{Theorem: Existence of nonradial solutions - 
Introduction} yields solutions with nonradial modulus.
With our methods, by combining \eqref{eq: Rotating solution ansatz} with a 
standing wave ansatz, we can also prove the existence of genuinely 
complex-valued 
ground states with nonradial modulus, see the appendix of this paper.

The paper is organized as follows. In Section~\ref{Section: Preliminaries}, we 
introduce Sobolev spaces via their spectral characterization and collect 
several known results on the properties of the zeros of Bessel functions. 
In Section~\ref{Section:Asymptotics of Zeros} we then prove a crucial technical 
estimate for certain sequences of such zeros.
This result is subsequently used in Section~\ref{Section: Spectral 
Characterization} 
to investigate the asymptotics of the zeros of Bessel functions in 
detail and, 
in particular, prove Theorem~\ref{Theorem: Spectrum Characterization - 
Introduction}.
Section~\ref{Section: Variational Formulation} is then devoted to the rigorous 
formulation of the variational framework outlined earlier and the proof of 
Theorem~\ref{Theorem: Existence of nonradial solutions - Introduction}. 
In Appendix~\ref{Appendix: Complex Case}, we discuss the results 
for complex-valued solutions mentioned above.
%
%
%
%%%%%%%%%%%%%%%%%%%%%%%%%%%%%%%%%%%%%%%%%%%%%%%%%%%%%%%%%%%%%%%%%%%%%%%%%%%%%%%%%%%%%%%%%%%%%%%%%%%%%%%%%%%%
%         ACKNOWLEDGMENTs
%
\subsection*{Acknowledgments}
The author thanks Tobias Weth for helpful discussions and comments.
%
%
%
%%%%%%%%%%%%%%%%%%%%%%%%%%%%%%%%%%%%%%%%%%%%%%%%%%%%%%%%%%%%%%%%%%%%%%%%%%%%%%%%%%%%%%%%%%%%%%%%%%%%%%%%%%%%
%
%
\section{Preliminaries}
\label{Section: Preliminaries}
We first collect some general facts on eigenvalues and eigenfunctions of the 
Laplacian on $\B$, we refer to \cite{Grebenkov-Nguyen} for a more comprehensive 
overview.
Recall that the eigenvalues of the problem 
\begin{equation*} % \label{eq: eigenvalue problem for Laplacian}
	\left\{ 
	\begin{aligned} 
		-\Delta u  & = \lambda u \quad && \text{in $\B$} \\
		u & = 0 && \text{on $\del \B$}
	\end{aligned}
	\right.
\end{equation*}
are given by $j_{\ell,k}^2$, where $j_{\ell,k}$ denotes the $k$-th zero of the 
Bessel function of the first kind $J_\ell$ with $k \in \N_0$, $l \in \N$. To 
each eigenvalue 
$j_{\ell,k}^2$ correspond two linearly independent eigenfunctions
\begin{equation}  \label{eq:eigenfunction-representation}
\begin{aligned}
	\phi_{\ell,k}(r,\theta)& \coloneqq A_{\ell,k} \cos(\ell \theta) J_\ell(j_{\ell,k} r) \\
	\psi_{\ell,k}(r,\theta)& \coloneqq B_{\ell,k} \sin(\ell \theta) 
	J_\ell(j_{\ell,k} r) ,
\end{aligned}
\end{equation}
where the constants $A_{\ell,k}, B_{\ell,k}>0$ are chosen such that $\|\phi_{\ell,k}(r,\theta)\|_2=\|\psi_{\ell,k}(r,\theta)\|_2=1$.
These functions constitute an orthonormal basis of $L^2(\B)$ and we can then 
characterize Sobolev spaces as follows:
$$
H^1_0(\B) \coloneqq \left\{ u \in L^2(\B): 
\|u\|_{H^1}^2\coloneqq\sum_{\ell=0}^\infty \sum_{k=1}^\infty j_{\ell,k}^2 
\left(|\langle u,\phi_{\ell,k} \rangle|^2 + |\langle u,\psi_{\ell,k} \rangle|^2 
\right) < \infty \right\} .
$$
It can be shown that this is consistent with the usual definition of $H^1(\B)$. By classical Sobolev embeddings, $H_0^1(\B)$ compactly maps into $L^p(\B)$ for any $1 \leq p < \infty$. 

Similarly, we consider the fractional Sobolev spaces 
$$
H^s_0(\B) \coloneqq \left\{ u \in L^2(\B): 
\|u\|_{H^s}^2\coloneqq\sum_{\ell=0}^\infty \sum_{k=1}^\infty j_{\ell,k}^{2s} 
\left(|\langle u,\phi_{\ell,k} \rangle|^2 + |\langle u,\psi_{\ell,k} \rangle|^2 
\right) < \infty \right\} 
$$
for $s \in (0,1)$. Using interpolation, it can be shown that this is equivalent 
to the classical definition and $H^s_0(\B)$ compactly maps into $L^p(\B)$ for 
$p< \frac{2}{1-s}$, i.e., there exists $C_s>0$ such that
\begin{equation*} % \label{eq: Fractional Sobolev inequality}
\|u\|_p \leq C_s \|u\|_{H^s_0(\B)}	
\end{equation*}
holds for $u \in H^1_0(\B)$.

Next, we collect several results on the properties of zeros Bessel functions, see e.g. \cite{Elbert-survey} for a more extensive overview.
In the following, we let $j_{\nu,k}$ denote the $k$-th zero of the Bessel function $J_\nu$, where $\nu \geq 0$, $k \in \N$. 
By definition, $j_{\nu,k}<j_{\nu,k+1}$. 
\begin{proposition} \label{Prop: Properties of Bessel function zeros}
	For each fixed $k \in \N$, $j_{\nu,k}$ is increasing with respect to $\nu$.
	Moreover, the following properties hold:
	\begin{itemize}
		\item[(i)] (\cite{Qu-Wong})
		We have
		\begin{equation*} 
			\nu + \frac{|a_k|}{2^\frac{1}{3}} \nu^\frac{1}{3} < j_{\nu,k} < \nu + \frac{|a_k|}{2^\frac{1}{3}} \nu^\frac{1}{3} + \frac{3}{20} |a_k|^2 \frac{2^\frac{1}{3}}{\nu^\frac{1}{3}}
		\end{equation*}
		where $a_k$ denotes the $k$-th negative zero of the Airy function $\mathrm{Ai}(x)$.
		\item[(ii)] (\cite{McCann}) 
		For each fixed $k \in \N$ the map 
		$$
		\nu \mapsto \frac{j_{\nu,k}}{\nu}
		$$
		is strictly decreasing on $(0,\infty)$.
		\item[(iii)] (\cite{Elbert-Laforgia: Further Results}) 
		For $k \in \N$ it holds that
		$$
		\pi k - \frac{\pi}{4}<j_{0,k} \leq \pi k - \frac{\pi}{4} + \frac{1}{8 \pi (k-\frac{1}{4})} .
		$$		
		\item[(iv)] (\cite{Elbert-Gatteschi-Laforgia}) 
		For each fixed $k \in \N$ the map $\nu \mapsto j_{\nu,k}$ is differentiable on $(0,\infty)$ and
		$$
		\frac{dj_{\nu,k}}{d\nu} \in \left(1,\frac{\pi}{2}\right) 
		$$
		for $\nu \geq 0$.
		% % % % % % % % % % % % % % % % % % % % % % % % % % % % % % % % % % % % % % % %
		%       NOTE
		% The upper bound seems to never be explicitly stated, but is an easy consequence of the Watson integral formula used below!
	\end{itemize} 
\end{proposition}
The zeros of the Airy function can in turn be estimated (see \cite{Breen}) by
\begin{equation*} % \label{eq: Breen estimate for Airy zeros}
\left( \frac{3\pi}{8} (4k-1.4) \right)^\frac{2}{3}<|a_k|< \left( \frac{3\pi}{8} (4k-0.965)\right)^\frac{2}{3} 
\end{equation*}
for $k \in \N$,
which yields the following result:
\begin{corollary}    \label{Corollary: Combined bounds}
	Let $j_{\nu,k} \in \R$ be defined as above. Then 	
	$$
	\nu+\frac{\left( \frac{3\pi}{8} (4k-2) \right)^\frac{2}{3}}{2^\frac{1}{3}} \nu^\frac{1}{3} 
	< j_{\nu,k}
	< \nu + \frac{\left( \frac{3\pi}{2} k\right)^\frac{2}{3}}{2^\frac{1}{3}} \nu^\frac{1}{3} + \frac{3}{20} \left( \frac{3\pi}{2} k\right)^\frac{4}{3} \frac{2^\frac{1}{3}}{\nu^\frac{1}{3}} .
	$$
\end{corollary}
%
%
%
%
%%%%%%%%%%%%%%%%%%%%%%%%%%%%%%%%%%%%%%%%%%%%%%%%%%%%%%%%%%%%%%%%%%%%%%%%%%%%%%%%%%%%%%%%%%%%%%%%%%%%%%%%%%%%
%
%
%
%
\section{Asymptotics of the Zeros of Bessel Functions}
\label{Section:Asymptotics of Zeros}

In order to study $L_\alpha$ in Section~\ref{Section: Spectral 
Characterization}, we will be particularly interested in the asymptotics of the 
zeros $j_{\nu,k}$ 
when the ratio ${\nu}/{k}$ remains fixed. For this case, we note the 
following 
result by Elbert and Laforgia:
\begin{theorem} \label{Theorem: Elbert-Laforgia Asymptotic Relation}
	{\bf(\cite{Elbert-Laforgia})} \\
	Let $x>-1$ be fixed. Then 
	$$
	\lim_{k \to \infty} \frac{j_{x k,k}}{k} \eqqcolon \iota(x) 
	$$
	exists. Moreover, $\iota(x)$ is given by
	$$
	\iota(x) = \begin{cases}
		\pi , \qquad & \text{$x=0$,} \\
		\frac{x}{\sin \phi} & \text{$x \neq 0$}
	\end{cases}
	$$
	where $\phi=\phi(x) \in [-\frac{\pi}{2},\frac{\pi}{2}]$ denotes the unique solution of
	\begin{equation} \label{eq: iota definition identity} 
		\frac{\sin \phi}{\cos \phi - (\frac{\pi}{2}-\phi) \sin \phi} = \frac{x}{\pi} .
	\end{equation}
\end{theorem}
Moreover, we note the following properties of a function associated to $\iota$.
\begin{lemma} \label{Lemma: Monotonicity and Inverse for iota quotient}
	The map 
	\begin{equation*} % \label{eq: decreasing iota quotient}
		f: (0,\infty) \to \R, \qquad f(x)= \frac{\iota(x)}{x}
	\end{equation*}
	is strictly decreasing and satisfies
	$$
	\lim_{x \to 0}  f(x) = \infty, \qquad \lim_{x \to \infty} f(x) = 1 .
	$$
	Moreover, its inverse is explicitly given by
	$$
	f^{-1}: (1,\infty) \to \R, \qquad f^{-1}(y) = \frac{\pi}{\sqrt{y^2- 1} - \left(\frac{\pi}{2} - \arcsin \frac{1}{y}\right) } .
	$$
\end{lemma}
\begin{proof}
	Note that the left hand side of \eqref{eq: iota definition identity} is 
	strictly 
	increasing with respect to $\phi$, and the right hand side is strictly 
	increasing 
	with respect to $x$, so that $\phi$ is necessarily an increasing function 
	of $x$.
	In particular, we then have $f(x)=\frac{\iota(x)}{x}= \frac{1}{\sin \phi}$ which clearly implies the monotonicity of $f$.
	
    Next, we note that $y=f(x)=\frac{1}{\sin \phi}$ implies $\phi = \arcsin \frac{1}{y}$ and hence
    \begin{align*}
    	\frac{x}{\pi} & = \frac{ \frac{1}{y} }{\cos \left(\arcsin \frac{1}{y}\right) - (\frac{\pi}{2}-\arcsin \frac{1}{y})  \frac{1}{y}} .
    \end{align*}
    The identity $\cos(\arcsin(t)) = \sqrt{1-t^2}$ then gives
    $$
    \frac{x}{\pi} =  \frac{ \frac{1}{y} }{\sqrt{1-\frac{1}{y^2}} - (\frac{\pi}{2}-\arcsin \frac{1}{y})  \frac{1}{y}}
    =\frac{1}{ \sqrt{y^2-1} - \left(\frac{\pi}{2} - \arcsin \frac{1}{y}\right)}
    $$
    and thus the claim follows.
\end{proof}
In order to characterize the eigenvalues of $L_\alpha$ later on, we need more 
information on the order of convergence in Theorem~\ref{Theorem: 
Elbert-Laforgia Asymptotic Relation}. To this end, we first recall some 
ingredients of the proof of this result. By  the Watson  integral  formula  
\cite[p. 508]{Watson}, for fixed $k \in \N$ the function $\nu \mapsto 
j_{\nu,k}$ satisfies %the nonlinear integro-differential equation
$$
\frac{d}{d\nu} j_{\nu,k} = 2 j_{\nu,k} \int_0^\infty K_0 (2  j_{\nu,k} 
\sinh(t)) e^{-2 \nu t} \, dt ,
$$
where $K_0$ denotes the modified Bessel function of the second kind of order zero. It then follows that the function 
$$
\iota_k(x) \coloneqq \frac{j_{kx,k}}{k}
$$
satisfies
\begin{equation} \label{eq: F_k definition}
	\frac{d}{dx} \iota_k(x)= 2 \iota_k \int_0^\infty K_0 \left( t 2 \iota_k  
	\frac{\sinh\left( \frac{t}{k} \right)}{\left(\frac{t}{k}\right)} \right) 
	e^{-2x t} \, dt \eqqcolon F_k(\iota_k,x) 
\end{equation}
for $k \in \N$ and $x \in (-1,\infty)$.
In \cite{Elbert-Laforgia} it is then shown that $\iota_k$ converges pointwise 
to the solution of 
\begin{equation} \label{eq: G definition}
	\left\{ 
	\begin{aligned} 
	\frac{d}{dx} \iota(x) &= 2 \iota \int_0^\infty K_0 \left( t 2 \iota  
	\right)  e^{-2x t} \, dt  \eqqcolon G(\iota,x) \\ %= \frac{\arccos 
	%\frac{x}{\iota(x)}}{\sqrt{1-\left(\frac{x}{\iota(x)}\right)^2}}\\
	\iota(0) &=\pi ,
	\end{aligned}
    \right.
\end{equation}
which is precisely given by the function $\iota$ discussed in 
Theorem~\ref{Theorem: Elbert-Laforgia Asymptotic Relation}. Moreover, it is 
shown that
\begin{equation} \label{eq: iota - iota_k inequality}
	\iota_k(x) < \iota(x)
\end{equation}
holds for all $k \in \N$.

We now give a more precise characterization of this convergence in the case 
$x>0$.
\begin{lemma} \label{Lemma: iota bounds}
	For any $x>0$ and $\eps>0$ there exists $k_0 \in \N$ such that
	$$
	- \exp \left( \left(\frac{1}{3} + \eps \right) x  \right) \frac{\pi}{4 k} \leq  \frac{j_{x k,k}}{k} - \iota(x)  \leq - (1-\eps) \frac{\pi}{4 k} 
	$$
	holds for $k \geq k_0$.
\end{lemma}
\begin{proof}
	Recall that we set $\iota_k(x)=  \frac{j_{x k,k}}{k}$ and the functions satisfy
	\begin{align*}
	\frac{d}{dx} \iota_k & = F_k(\iota_k,x) \\
	\frac{d}{dx} \iota & = G(\iota,x)
	\end{align*}
	in $(-1,\infty)$ with $F_k$ and $G$ defined in \eqref{eq: F_k definition} and \eqref{eq: G definition}, respectively. Now consider $u_k(x)\coloneqq\iota_k(x)-\iota(x)$ so that
	$$
	\frac{d}{dx} u_k = \frac{F_k(\iota_k,x)-G(\iota,x)}{\iota_k(x)-\iota(x)} u_k(x) = \beta_k(x) u_k(x) 
	$$
	where we set 
	$$
	\beta_k(x)\coloneqq\frac{F_k(\iota_k,x)-G(\iota,x)}{\iota_k(x)-\iota(x)} .
	$$
	Note that $\beta_k$ is well-defined by \eqref{eq: iota - iota_k inequality}. 
	In particular, we find that
	$$
	u_k(x)= u_k(0) \exp \left( \int_0^x \beta_k(t) \, dt \right) .
	$$
	Next, we note that the monotonicity of $K_0$ and the fact that $\sinh(t)>t$ holds for $t>0$ imply
	\begin{equation*} %\label{eq: iota ODE Nonlinearity comparison} 
	\begin{aligned} 
	F_k(\iota_k,x) &  = 2 \iota_k \int_0^\infty K_0 \left( t 2 \iota_k  
	\frac{\sinh\left( \frac{t}{k} \right)}{\left(\frac{t}{k}\right)} \right) 
	e^{-2x t} \, dt \\
	&< 2 \iota \int_0^\infty K_0 \left( t 2 \iota  \right)  
	e^{-2x t} \, dt =G(\iota_k,x) %= \frac{\arccos 
	%\frac{x}{\iota_k(x)}}{\sqrt{1-\left(\frac{x}{\iota_k(x)}\right)^2}}
    \end{aligned}
	\end{equation*}
    where \cite[p. 388]{Watson} implies
    \begin{equation} \label{eq:G-identity}
     G(y,x) = 2 y \int_0^\infty K_0 \left( t 2 y  \right)  e^{-2x t} \, dt  = 
     \frac{\arccos \frac{x}{y}}{\sqrt{1-\left(\frac{x}{y}\right)^2}} \qquad 
     \text{if $ \left|\frac{x}{y}\right|<1$.}
    \end{equation}
%    provided $ |\frac{x}{y}|<1$.
    Importantly, the function 
	$$
	g:(1,\infty) \mapsto \R, \quad t \mapsto \frac{\arccos \frac{1}{t}}{\sqrt{1-\frac{1}{t^2}}} 
	$$
	is strictly increasing. Indeed, we have
	\begin{align*}
		g'(t) & = \frac{1}{t^2(1-\frac{1}{t^2})} - \frac{\arccos \frac{1}{t}}{t^3 (1-\frac{1}{t^2})^\frac{3}{2}}
		= \frac{1}{t^2 (1-\frac{1}{t^2})} \left( 1 - \frac{\arccos 
		\frac{1}{t}}{t \sqrt{1-\frac{1}{t^2}}} \right) \\
		& 	=\frac{1}{t^2 -1} \left( 1 - \frac{\arccos \frac{1}{t}}{ 
		\sqrt{t^2-1}} \right)
	\end{align*} 
    and \cite[Theorem 2 for $b={1}/{2}$]{Zhao-Wei-Guo-Qi} gives
    $$
    \arccos s < 2 \frac{\sqrt{1-s}}{\sqrt{1+s}}
    $$
    for $s \in (0,1)$ so that
    $$
    \frac{\arccos \frac{1}{t}}{ \sqrt{t^2-1}} = \frac{\arccos \frac{1}{t}}{t \sqrt{1-\frac{1}{t}} \sqrt{1+\frac{1}{t}}} <  \frac{2}{t \sqrt{1+\frac{1}{t}}} = \frac{2}{\sqrt{t^2+t}} <1
    $$
    holds for $t >1$, which implies that $g'$ is a positive function. Moreover, $g'$ can be continuously extended by $g'(1)=\frac{1}{3}$ and is decreasing, which implies 
    \begin{equation} \label{eq: g' bound}
    g'(t) \leq \frac{1}{3}
    \end{equation}
    for $t>1$.
    
    Noting that Lemma~\ref{Lemma: Monotonicity and Inverse for iota quotient} 
    and
    the convergence $\iota_k(x) \to \iota(x)$ imply that 
    $\left|\frac{x}{\iota_k(x)}\right|<1$ holds for sufficiently large $k$, we 
    may combine the identity~\eqref{eq:G-identity} with  $\iota_k(x) < 
    \iota(x)$ and the monotonicity properties stated above to deduce 
    $F_k(\iota_k,x) < G(\iota_k,x)< G(\iota,x)$ and hence 
	\begin{equation} \label{eq: beta estimates} 
	0 \leq \beta_k(x) .% \leq \exp \left( \frac{x}{3}  \right) .
    \end{equation}
    In particular, this yields
    $$
    u_k(x) \leq u_k(0)
    $$
    for $x>0$. 
    
	We now estimate $u_k(0)$. Recall that $\iota(0)=\pi$ and therefore
	$$
	u_k(0)=  \frac{j_{0,k}}{k} - \pi ,
	$$
	so Proposition~\ref{Prop: Properties of Bessel function zeros}(iii) yields
	\begin{equation}  \label{eq: u_k(0) estimates}
	- \frac{\pi}{4 k} \leq u_k(0) \leq - \frac{\pi}{4 k} + \frac{1}{8 \pi k (k-\frac{1}{4})} .
    \end{equation}
	In view of $u_k(x)= u_k(0) \exp \left( \int_0^x \beta_k(t) \, dt \right)$, 
	combining the last estimate with \eqref{eq: beta estimates} therefore 
	implies
	$$
	u_k(x) \leq - \frac{\pi}{4 k} + \frac{1}{8 \pi k (k-\frac{1}{4})}
	$$
	for $x>0$ and hence the upper bound stated in the claim.
	
	It remains to prove the lower bound. To this end, we employ arguments inspired by \cite{Bobkov}
	and first note that
	$$
	\sinh(x) \leq x+x^3
	$$
	holds for $x \in (0,1)$, which implies
	\begin{equation} \label{eq: sinh upper bound}
	\frac{\sin \left(\frac{t}{k}\right)}{\frac{t}{k}} \leq 1 + \frac{1}{k^\frac{4}{3}}
    \end{equation}
	for $k \in \N$ and $0 < t \leq t_k \coloneqq k^\frac{1}{3}$. In the following, we fix $x>0$ and let $y>x$. Then the monotonicity of $K_0$ and \eqref{eq: sinh upper bound} yield
	\begin{align*}
		F_k(y,x)& \geq \int_0^{t_k} K_0 \left(  t \left( 1 + 
		\frac{1}{k^\frac{4}{3}} \right)\right) e^{-\frac{x t}{y}} \, dt
	\end{align*}
    and therefore
    \begin{equation} \label{eq:difference-estimate-lower}
    \begin{aligned}
    	    	& F_k(y,x) - G(y,x) \\
         \geq & \int_0^{t_k} \left[ K_0 \left(  t \left( 1 + 
    	\frac{1}{k^\frac{4}{3}} \right)\right) - K_0(t) \right] e^{-\frac{x 
    	t}{y}} \, dt - \int_{t_k}^\infty K_0(t) e^{-\frac{x t}{y}} \, dt .
    \end{aligned}
    \end{equation}
    From
    $$                      % % This can be deduced directly from the integral characterization
    K_0(t) \leq K_\frac{1}{2}(t)=\sqrt{\frac{\pi}{2t}} e^{-t}       
    $$
    we then find 
    \begin{align*}
    	\int_{t_k}^\infty K_0(t) e^{-\frac{x t}{y}} \, dt & \leq 
    	\sqrt{\frac{\pi}{2t_k}} e^{-t_k} \int_0^\infty e^{-\frac{x t}{y}} \, dt 
    	= \sqrt{\frac{\pi}{2t_k}} e^{-t_k} \frac{y}{x} .
    \end{align*}
    For any $y_0>x$ and $\delta \in (0,1)$, we thus find $k_0 \in \N$ such that
    \begin{equation} \label{eq: Estimate for second integral}
    \int_{t_k}^\infty K_0(t) e^{-\frac{x t}{y}} \, dt  \leq 
    \frac{\delta}{k^\frac{4}{3}}
    \end{equation}
    holds for $|y-y_0|<y_0-x$ and $k \geq k_0$.
    
    In order to estimate the other term in 
    \eqref{eq:difference-estimate-lower}, we note that for $t \in \R$ there 
    exists $\xi_k \in \left(t,t+\frac{t}{k^\frac{4}{3}}\right)$ such that
    $$
    K_0 \left(  t \left( 1 + \frac{1}{k^\frac{4}{3}} \right)\right) - K_0(t) = K_0'(\xi_k) \frac{1}{k^\frac{4}{3}} = - K_1(\xi_k) \frac{t}{k^\frac{4}{3}} \geq -K_1(t) \frac{t}{k^\frac{4}{3}} .
    $$
    This implies
    \begin{align*}
    	\int_0^{t_k} \left[ K_0 \left(  t \left( 1 + \frac{1}{k} \right)\right) 
    	- K_0(t) \right] e^{-\frac{x t}{y}} \, dt 
    	& \geq  - \frac{1}{k^\frac{4}{3}} \int_0^{t_k} K_1(t) t 
    	e^{-\frac{xt}{y}} \, dt \\
    	& \geq - \frac{1}{k^\frac{4}{3}} \int_0^{\infty} 
    	K_1(t) t e^{-\frac{xt}{y}} \, dt ,
    \end{align*} 
    where \cite[p. 388]{Watson} gives
    $$
     \int_0^{\infty} K_1(t) t e^{-\frac{xt}{y}} \, dt \leq  \int_0^{\infty} 
     K_1(t) t  \, dt = 
     \Gamma\left(\frac{1}{2}\right)\Gamma\left(\frac{3}{2}\right) = 
     \frac{\pi}{2} .
    $$
    Combined with \eqref{eq: Estimate for second integral}, it thus follows 
    that for any $x>0$, $y_0>x$ and $\delta \in (0,1)$ there exists $k_0 \in 
    \N$ such that
    \begin{align*}
    	F_k(y,x) - G(y,x) & \geq - \left( \frac{\pi}{2} + \delta \right) \frac{1}{k^\frac{4}{3}}
    \end{align*}
    holds for $k \geq k_0$ and $|y-y_0|<y_0-x$. 
    
    We now proceed by taking $y_0 = \iota(x)$ and note that there exists $k_0' 
    \in \N$ such that $|\iota_k(x)-\iota(x)|<\iota(x)-x$ holds for $k \geq 
    k_0'$. Combined with \eqref{eq: g' bound}, we then conclude that for given 
    $\delta \in (0,1)$ we can find $k_0 \in \N$ such that
    \begin{align*}
    F_k(\iota_k,x)-G(\iota,x) & = F_k(\iota_k,x) - G(\iota_k,x) + G(\iota_k,x) - G(\iota,x) \\
    & \geq    - \left( \frac{\pi}{2} + \delta \right) \frac{1}{k^\frac{4}{3}}  - \max_{\xi \in (\iota_k(x),\iota(x))} \frac{dG}{dy}(\xi,x) |\iota_k(x) - \iota(x)| \\
    & \geq   - \left( \frac{\pi}{2} + \delta \right) \frac{1}{k^\frac{4}{3}}   
    - \frac{1}{3} |\iota_k(x) - \iota(x)| 
 	\end{align*}
    holds for $k \geq k_0$.
    It follows that
    \begin{align*}	
    \beta_k(x) & = \frac{F_k(\iota_k,x)-G(\iota,x)}{\iota_k(x)-\iota(x)} 
    \leq \frac{1}{3} +   \left( \frac{\pi}{2} + \delta \right) \frac{1}{k^\frac{4}{3}} \frac{1}{|\iota_k(x) - \iota(x)|} 
    \end{align*}
    and since $|\iota_k(x) - \iota(x)| = |u_k(x)| \geq (\frac{\pi}{4} - \delta) 
    \frac{1}{k}$ holds for sufficiently large $k$ we therefore have
    \begin{align*}
    	\beta_k(x) & \leq \frac{1}{3} +  \frac{\frac{\pi}{2} + \delta}{\frac{\pi}{4}-\delta} \frac{1}{k^\frac{1}{3}} .
    \end{align*} 
    Consequently, we may choose $k_0$ such that 
    $$
    \beta_k(x)  \leq \frac{1}{3} + \eps 
    $$
    holds for $k \geq k_0$. Overall, this yields
    $$
    1 \leq \exp \left( \int_0^x \beta_k(t) \, dt \right) \leq \exp \left( 
    \left(\frac{1}{3} + \eps \right) x  \right)
    $$
    for $k \geq k_0$. Recalling \eqref{eq: u_k(0) estimates} and
    $$
    u_k(x)= u_k(0) \exp \left( \int_0^x \beta_k(t) \, dt \right), 
    $$
    the claim thus follows.
\end{proof}
\begin{remark} \label{Remark: iota bounds estimate}
	    Lemma~\ref{Lemma: iota bounds} improves the bound obtained in 
	    \cite[Theorem 2.1]{Elbert-Laforgia} as follows:
	    
	    For any $\eps,\nu >0$ there exists $k_0 \in \N$ such that 
	    $$
	    j_{\nu,k} < k \, \iota \left( \frac{\nu}{k} \right) - (1-\eps) 
	    \frac{\pi}{4} 
	    $$
	    holds for $k \geq k_0$.
\end{remark}
%
%
%
%
%
%%%%%%%%%%%%%%%%%%%%%%%%%%%%%%%%%%%%%%%%%%%%%%%%%%%%%%%%%%%%%%%%%%%%%%%%%%%%%%%%%%%%%%%%%%%%%%%%%%%%%%%%%%%%%
%
%
%
%
\section{Spectral Characterization} 
\label{Section: Spectral Characterization}
For $\alpha>1$ recall the operator 
$$
L_\alpha=-\Delta + \alpha^2 \del_{\theta}^2.
$$
If $\phi \in H^1_0(\B)$ is an eigenfunction of $-\Delta$ corresponding to the 
eigenvalue $j_{\ell,k}^2$, then it follows from the representation 
\eqref{eq:eigenfunction-representation} that $\phi$ is also an eigenfunction of 
$L_\alpha$ with
$$
L_\alpha \phi = (j_{\ell,k}^2-\alpha^2 \ell^2) \phi .
$$
Since the eigenfunctions of $-\Delta$ constitute an orthonormal basis of $L^2(\B)$, we find that the Dirichlet eigenvalues of $L_\alpha$ are given by 
$$
\left\{ j_{\ell,k}^2-\alpha^2 \ell^2 : \ \ell \in \N_0, \, k \in \N \right\}.
$$
In the following, we wish to study this set in more detail. The following 
result already shows a stark contrast to the case $\alpha \in [0,1]$.
\begin{proposition} % \label{Proposition: Spectrum unbounded above and below}
	Let $\alpha>1$. Then the spectrum of the operator $L_\alpha=-\Delta + 
	\alpha^2 \del_\theta^2$ is unbounded from above and below.
\end{proposition}
\begin{proof}
	For $\ell \in \N_0$, $k \in \N$ we write
	$$
	j_{\ell,k}^2-\alpha^2 \ell^2=(j_{\ell,k}+\alpha \ell) \ell \left(\frac{j_{\ell,k}}{\ell}-\alpha\right)
	$$ 
	and note that Corollary~\ref{Corollary: Combined bounds} implies
	\begin{equation} \label{Upper and lower bound}
		1-\alpha+\frac{\left( \frac{3\pi}{8} (4k-2) \right)^\frac{2}{3}}{2^\frac{1}{3}} \ell^{-\frac{2}{3}} <
		\frac{j_{\ell,k}}{\ell}-\alpha 
		< 1-\alpha + \frac{\left( \frac{3\pi}{2} k\right)^\frac{2}{3}}{2^\frac{1}{3}} \ell^{-\frac{2}{3}} + \frac{3}{20} \left( \frac{3\pi}{2} k \right)^\frac{4}{3} \frac{2^\frac{1}{3}}{\ell^\frac{4}{3}} .
	\end{equation}
    If we choose sequences $(\ell_i)_i$, $(k_i)_i$, such that 
    $\frac{\ell_i}{k_i} \to \infty$, this readily implies 
    $j_{\ell_i,k_i}^2-\alpha^2 \ell_i^2 \to -\infty$, whereas sequences such 
    that $\frac{\ell_i}{k_i} \to 0$ yield $j_{\ell_i,k_i}^2-\alpha^2 \ell_i^2 
    \to 
    \infty$ and thus the claim.
\end{proof}
In particular, this proves the first part of Theorem~\ref{Theorem: Spectrum 
Characterization - Introduction}.
As noted in the introduction, it is not clear whether the spectrum of 
$L_\alpha$ only consists of isolated points. Indeed, note that 
Lemma~\ref{Lemma: iota bounds} suggests that certain subsequences of 
$j_{\ell,k} - \alpha \ell$ may converge and it is unclear if there exists a 
subsequence that even converges to zero. In particular, the spectrum of the 
operator could even be dense in $\R$.

This is excluded by the second part of Theorem~\ref{Theorem: Spectrum 
Characterization - Introduction} which we restate as follows.
\begin{theorem}  \label{Theorem: Spectrum has no accumulation point}
	 There exists a sequence $(\alpha_n)_n  \subset (1,\infty)$ 
	 such that the following properties hold for $n \in \N$:
	 \begin{itemize}
	 	\item[(i)] The spectrum of $L_{\alpha_n}$ consists of eigenvalues with 
	 	finite multiplicity.
	 	\item[(ii)] There exists $c_n>0$ 
	 	such that for each 
	 	$\ell \in \N_0$, $k \in \N$ we either have $j_{\ell,k}^2-\alpha_n^2 
	 	\ell^2=0$ or
	 	\begin{equation} \label{eq: eigenvalue lower bound-restated} 
	 	|j_{\ell,k}^2-\alpha_n^2 \ell^2| \geq c_n j_{\ell,k}.
	 	\end{equation} 
	 	\item[(iii)] The spectrum of $L_{\alpha_n}$ has no finite 
	 	accumulation points.
	 \end{itemize}
\end{theorem}
\begin{proof}
	We set $\sigma_n := \frac{1}{n}$ and $\alpha_n := 
	\frac{\iota(\sigma_n)}{\sigma_n}$, where the function $\iota$ is given by 
	Theorem~\ref{Theorem: Elbert-Laforgia Asymptotic Relation}. It then 
	suffices to show that there exists $n_0 \in \N$ such that properties 
	(i)-(iii) hold for $n \geq n_0$. In the following, we fix $n \in \N$ and  
	assume that there 
	exists $\Lambda \in \R$ and increasing sequences $(\ell_i)_i$, $(k_i)_i$ 
	such that $j_{\ell_i,k_i}^2-\alpha_n^2 \ell_i^2 \to \Lambda$ as $n \to 
	\infty$. 
	Note that the case of an eigenvalue with infinite multiplicity, i.e.,  
	$j_{\ell_i,k_i}^2-\alpha_n^2 \ell_i^2 = \Lambda$ for all $i$, 
	is included here. 	
	The identity $j_{\ell,k}^2-\alpha_n^2 \ell^2=(j_{\ell,k}+\alpha_n \ell) 
	\ell \left(\frac{j_{\ell,k}}{\ell}-\alpha_n\right)$ then implies that we 
	must have
	\begin{equation} \label{eq:quotient-convergence}
	\frac{j_{\ell_i,k_i}}{\ell_i} \to \alpha_n .
	\end{equation}
	Our goal is to show that such sequences can only converge of order 
	$\frac{1}{l_i}$, which will allow us to derive a suitable contradiction.
	
	Firstly, the estimate \eqref{Upper and lower bound} implies that there must 
	exist $\sigma \in (0,\infty)$ such that $\frac{\ell_i}{k_i} \to \sigma$. 	
	We now claim that for any $\eps>0$, there exists $i_0 \in \N$ such that
	\begin{equation}  \label{eq: eps-inequality for quotient}
		(1-\eps) \frac{j_{\sigma k_i,k_i}}{\sigma k_i} < 	
		\frac{j_{\ell_i,k_i}}{\ell_i} < (1+\eps) \frac{j_{\sigma 
		k_i,k_i}}{\sigma k_i} \qquad \text{for $i \geq i_0$.}
	\end{equation}
	To this end, we first assume that
	$\ell_i < \sigma k_i$ holds. Then		
	Proposition~\ref{Prop: Properties of Bessel function zeros}(ii) implies
	$$
	\frac{j_{\ell_i,k_i}}{\ell_i}   >  \frac{j_{\sigma k_i,k_i}}{\sigma k_i}     
	$$
	and, in particular, the lower bound. Moreover, the fact that the function 
	$\nu \mapsto j_{\nu,k}$ is increasing for fixed $k$ yields
	$$
	\frac{j_{\ell_i,k_i}}{\ell_i} \leq 	\frac{j_{\sigma k_i,k_i}}{\ell_i} = \frac{\sigma k_i}{\ell_i} \frac{j_{\sigma k_i,k_i}}{\sigma k_i} . 
	$$
	Noting that $\frac{\sigma k_i}{\ell_i} \to 1$ as $i \to \infty$, we 
	conclude that for any $\eps>0$, there exists $i_0 \in \N$ such that
	$$
	\frac{j_{\ell_i,k_i}}{\ell_i} < (1+\eps) \frac{j_{\sigma k_i,k_i}}{\sigma k_i}
	$$
	holds for $i \geq i_0$ with $\ell_i < \sigma k_i$. The case $\ell_i 
	\geq \sigma k_i$ can be treated analogously. 
	
	Overall, \eqref{eq: eps-inequality for quotient} implies
	$$
	(1-\eps)\frac{\iota(\sigma)}{\sigma} \leq 	\liminf_{n \to \infty} 
	\frac{j_{\ell_i,k_i}}{\ell_i} \leq  	\limsup_{n \to \infty} 
	\frac{j_{\ell_i,k_i}}{\ell_i} \leq (1+\eps)  \frac{\iota(\sigma)}{\sigma} 
	$$
	for arbitrary $\eps>0$, with the function $\iota$ given by 
	Theorem~\ref{Theorem: Elbert-Laforgia 
	Asymptotic Relation}. In particular, \eqref{eq:quotient-convergence} then 
	yields
	$$
	\frac{\iota(\sigma)}{\sigma} = 	\lim_{n \to \infty} 
	\frac{j_{\ell_i,k_i}}{\ell_i} 
	= \alpha_n
	$$ 
	and Lemma~\ref{Lemma: Monotonicity and Inverse for iota quotient} thus 
	implies that we must have $\sigma = \sigma_n$ due to our choice of 
	$\alpha_n$. In particular, it follows that $\frac{\ell_i}{k_i} \to 
	\sigma_n$. 
	We now distinguish two cases: 
	
	\textbf{Case 1:} There exists $i_0 \in \N$ such that $\frac{\ell_i}{k_i} 
	\geq \sigma_n$ holds for $i \geq i_0$. \\
	In this case, Proposition~\ref{Prop: Properties of Bessel function zeros}(ii) implies
	$$
	\frac{j_{\ell_i,k_i}}{\ell_i}-\alpha_n \leq \frac{j_{\sigma_n 
	k_i,k_i}}{\sigma_n k_i}-\alpha_n = \frac{1}{\sigma_n} 
	\left(\frac{j_{\sigma_n k_i,k_i}}{k_i}- \iota(\sigma_n) \right) ,
	$$ 
	so that Lemma~\ref{Lemma: iota bounds} yields 
	$$
	\frac{j_{\ell_i,k_i}}{\ell_i}-\alpha_n \leq - \frac{\pi}{8\sigma_n k_i} 
	$$
	for $i \geq i_0$, after possibly enlarging $i_0$. In particular, this 
	implies 
	$$
	|j_{\ell_i,k_i} - \alpha_n \ell_i| = \ell_i \left| 
	\frac{j_{\ell_i,k_i}}{\ell_i}-\alpha_n \right| \geq \frac{\pi 
	\ell_i}{8\sigma_n 
	k_i} \geq \frac{\pi}{8}
	$$
	for $i \geq i_0$ and therefore $\liminf_{i \to \infty} 
	(j_{\ell_i,k_i}-\alpha_n \ell_i) \geq \frac{\pi}{8}$.
    
    \medskip
    \textbf{Case 2:} There exists $i_0 \in \N$ such that $\frac{\ell_i}{k_i} < 
    \sigma_n$ holds for $i \geq i_0$. \\
    We may write $\ell_i=\sigma_n k_i - \delta_i$ with $\delta_i>0$ satisfying 
    $\frac{\delta_i}{k_i} \to 0$ as $i \to \infty$. Then
    \begin{align*}
    	j_{\ell_i,k_i}-\alpha_n \ell_i & =
    	j_{(\sigma_n k_i - \delta_i),k_i} - \alpha_n (\sigma_n k_i-\delta_i)  \\
    	&= (j_{\sigma_n k_i,k_i} -\alpha_n \sigma_n k_i) + (j_{(\sigma_n k_i - 
    	\delta_i),k_i} - j_{\sigma_n k_i ,k_i}) + \alpha_n \delta_i \\
    	& =    (j_{\sigma_n k_i,k_i} -\alpha_n \sigma_n k_i)  + R_{n,i} 
    	\delta_i ,	
    \end{align*}
    where we have set
    $$
    R_{n,i} \coloneqq \alpha_n - \frac{j_{\sigma_n k_i ,k_i}-j_{\sigma_n k_i - 
    \delta_i ,k_i}}{\delta_i} .
    $$
    By Lemma~\ref{Lemma: iota bounds} we may further enlarge $i_0$ to ensure 
    that
    $$
    j_{\sigma_n k_i,k_i} -\alpha_n \sigma_n k_i \geq -
    \frac{\pi}{4} e^{\sigma_n/3}
    %  \exp \left( \frac{\sigma_n}{3} \right) 
    $$
	holds for $i \geq i_0$.
    Next, Proposition~\ref{Prop: Properties of Bessel function zeros}(iv) gives
    $$
    \frac{ j_{\sigma_n k_i ,k_i} - j_{\sigma_n k_i-\delta_i ,k_i}}{\delta_i} 
    \in \left(1, \frac{\pi}{2}\right) 
    $$
    and hence
     \begin{align*} 
    	\liminf_{i \to \infty} R_{n,i} & \geq \alpha_n - \frac{\pi}{2}.
    \end{align*}
	Since $\alpha_n = \frac{\iota(\sigma_n)}{\sigma_n} \to \infty$ as $n \to 
	\infty$ by Lemma~\ref{Lemma: Monotonicity and Inverse for iota quotient}, 
	this term is positive for sufficiently large $n$, and it therefore 
	follows that
    \begin{align*}
    	\liminf_{i \to \infty} (j_{\ell_i,k_i}-\alpha_n \ell_i) & \geq 
    	-\frac{\pi}{4} e^{\sigma_n/3} 
    	+ 
    	\left(\alpha_n - \frac{\pi}{2} \right) \inf_{i \in \N} \delta_i .
    \end{align*}
    In order to show that the right hand side is positive, we recall that 
    $\sigma_n = \frac{1}{n}$ and therefore the fact that $\ell_i=\sigma_n k_i - 
    \delta_i$ must be a natural number implies $\delta_i = \frac{n'}{n}$ 
    for some $n' \in 
    \N$ and, in particular, $ \inf_i \delta_i = \frac{1}{n}$. 
    Moreover, by 
    Lemma~\ref{Lemma: Monotonicity and Inverse for iota quotient} the 
    associated $\alpha_n = \frac{\iota(\sigma_n)}{\sigma_n}$ is uniquely 
    determined by the equation
    $$
    \pi n  = \frac{\pi}{\sigma_n}= \sqrt{\alpha_n^2- 1} - \left(\frac{\pi}{2} - 
    \arcsin \frac{1}{\alpha_n}\right) .
    $$
    Since the right hand side is strictly increasing in $\alpha_n$ and we have
    \begin{align*} 
    \frac{\sqrt{n^2- 1} - \left(\frac{\pi}{2} - \arcsin \frac{1}{n}\right)}{n} & = \sqrt{1-\frac{1}{n^2}} + \frac{1}{n} \arcsin \frac{1}{n} - \frac{\pi}{2n} \to 1 < \pi
    \end{align*}
    as $n \to \infty$, there must exist $n_0 \in \N$ such that $\alpha_n > n$ 
    holds for $n \geq n_0$.
    We thus have 
    \begin{align*} 
    	-\frac{\pi}{4} e^{\sigma_n/3} + \left(\alpha_n - 
    	\frac{\pi}{2} \right) \inf_i \delta_i 
    	& \geq -\frac{\pi}{4} e^{\sigma_n/3} + \frac{1}{n} \left(n - 
    	\frac{\pi}{2} 
    	\right) \\
    	& = 1 - \pi \left( \frac{1}{2n} + \frac{ e^\frac{1}{3n}}{4} \right) \to 
    	1 - \frac{\pi}{4}> 0
    \end{align*}
    as $n \to \infty$. We conclude that after possibly further enlarging $n_0$, 
    $$
    \kappa_n:=-\frac{\pi}{4} e^{\sigma_n/3} + \left(\alpha_n - \frac{\pi}{2} 
    \right) 
    \inf_i \delta_i >0
    $$
    holds for $n \geq n_0$.
    % and therefore $\liminf_{i \to \infty}     (j_{\ell_i,k_i}-\alpha_n 
    %\ell_i) \geq C_n > 0$.
    
    \medskip
    Since we may always pass to a subsequence such that one of these two 
    cases holds, we overall find that
    $$
    \liminf_{i \to \infty} (j_{\ell_i,k_i}-\alpha_n \ell_i) \geq 
    \min\left\{\kappa_n, \frac{\pi}{4}\right\}>0
    $$
    for any sequences $(\ell_i)_i,(k_i)_i$ such that $\frac{\ell_i}{k_i} 
    \to \sigma_n = \frac{1}{n}$, provided $n \geq n_0$. In particular, it 
    follows that $j_{\ell,k}^2-\alpha_n^2 \ell^2 = (j_{\ell,k}-\alpha_n \ell)(
    j_{\ell,k}+\alpha_n \ell)$ cannot converge to $\Lambda$, which implies (i) 
    and (iii). 
    
    Moreover, since we considered arbitrary sequences satisfying 
    \eqref{eq:quotient-convergence}, we further find that
    $$
    \gamma_n\coloneqq\lim_{N \to \infty} \inf_{\ell,k \geq N} |j_{\ell,k}-\alpha_n \ell| > 0
    $$
    for $n \geq n_0$. Consequently, taking $N_0 \in \N$ such that $\inf_{\ell,k 
    \geq N} 
    |j_{\ell,k}-\alpha_n \ell| > \frac{\gamma_n}{2}$ holds for $N \geq N_0$ and 
    setting 
    $$
    c_n \coloneqq \min \left\{ \frac{\gamma_n}{2} ,  \ \inf_{\substack{\ell,k \leq N_0 \\ j_{\ell,k} \neq \alpha_n \ell}} |j_{\ell,k}-\alpha_n \ell| \right\} >0 
    $$
    yields 
    $$
    |j_{\ell,k}^2-\alpha_n^2 \ell^2| = |j_{\ell,k}-\alpha_n \ell| 
    |j_{\ell,k}+\alpha_n \ell| \geq c_n j_{\ell,k}
    $$
    as claimed in (ii). This completes the proof.
\end{proof}
\begin{remark}
	 \begin{itemize} 
        \item[(i)]
        The sequence $(\alpha_n)_n$ can be characterized further by noting that
        $$
        \pi n  = \sqrt{\alpha_n^2- 1} - \left(\frac{\pi}{2} - \arcsin 
        \frac{1}{\alpha_n}\right) 
        $$
        implies
        $$
        \alpha_n^2 = 1 + \left( \pi n + \frac{\pi}{2} - \arcsin 
        \frac{1}{\alpha_n} \right)^2 .
        $$
        Since $\arcsin\frac{1}{\alpha_n} = O(n^{-2})$, this implies 
        $$
        \alpha_n^2 \approx 1 + \left( \pi n + \frac{\pi}{2}  \right)^2
        $$
        	 \item[(ii)]
        	 The methods used above can be further extended to include some 
        	 additional 
        	 values of $\alpha$. If we let $\sigma=\frac{m}{n}$ with $m,n \in 
        	 \N$, we 
        	 find that $\inf_i \delta_i = \frac{1}{n}$ and similar arguments as 
        	 above then lead to the condition
        	 $$
        	 0< \sqrt{\frac{1}{n^2} + \frac{\pi^2}{m^2}} - \pi \left( 
        	 \frac{1}{2n} + 
        	 \frac{e^\frac{m}{3n}}{4} \right) .
        	 $$
        	 As $n \to \infty$, we find that this holds for $m=1,2,3$.
        	 
        	 Moreover, we note that numerical computations imply that the 
        	 result 
        	 should also hold for $\sigma=1,2,3$.
	\end{itemize}
\end{remark}
%
%
%
%
%
%
%
%%%%%%%%%%%%%%%%%%%%%%%%%%%%%%%%%%%%%%%%%%%%%%%%%%%%%%%%%%%%%%%%%%%%%%%%%%%%%%%%%%%%%%%%%%%%%%%%%%%%%%%%%%%%%
%
%
\section{Variational Characterization of Ground States}
\label{Section: Variational Formulation}
We now return to solutions of \eqref{Reduced equation}.
Setting
$$
L_{\alpha,m} \coloneqq -\Delta + \alpha^2 \del_{\theta}^2 + m
$$
for $\alpha>1$, $m \in \R$, our first goal is to find a suitable domain for the 
quadratic form
$$
u \mapsto \langle L_{\alpha,m} u,u \rangle_{L^2(\B)} = \int_\B (L_{\alpha,m} u) 
u \, dx = \int_{\B }\left(|\nabla u|^2-\alpha^2 |\del_\theta u|^2 + m 
u^2\right) \, dx .
$$
In order to simplify the notation, we set 
$$
\begin{aligned} 
\cI_{\alpha,m}^+ & \coloneqq \left\{ (\ell,k) \in \N_0 \times \N: j_{\ell,k}^2 - \alpha^2 \ell^2 +m>0  \right\} \\
\cI_{\alpha,m}^0 & \coloneqq \left\{ (\ell,k) \in \N_0 \times \N: j_{\ell,k}^2 - \alpha^2 \ell^2 +m=0  \right\} \\
\cI_{\alpha,m}^- & \coloneqq \left\{ (\ell,k) \in \N_0 \times \N: j_{\ell,k}^2 - \alpha^2 \ell^2 +m<0  \right\}
\end{aligned}
$$
for $\alpha>1$, $m \in \R$, i.e., the index sets of positive, zero and negative 
eigenvalues, respectively.
Instead of restricting ourselves to the sequence $(\alpha_n)$ given by 
Theorem~\ref{Theorem: Spectrum Characterization - Introduction}, we consider
$$
\cA\coloneqq \left\{ \alpha>1: \ \text{
	$|\cI_{\alpha,0}^0|<\infty$ \ and $\min_{(\ell,k) \not \in 
	\cI_{\alpha,0}^0} |j_{\ell,k} - \alpha \ell|>0$} \right\} .
%	there exists $c_\alpha>0$ such that $|j_{\ell,k} - \alpha \ell| \geq c_\alpha$ for  $(\ell,k) \not \in \cI_{\alpha,0}$}\right\}.
$$
In particular, $\cA$ contains the sequence $(\alpha_n)_n$ and is therefore 
nonempty and unbounded. Moreover, writing $j_{\ell,k}^2-\alpha^2 \ell^2= 
(j_{\ell,k}+\alpha \ell)(j_{\ell,k}-\alpha \ell)$ we find that for any $\alpha 
\in\cA$ there exists 
$c_\alpha>0$ such that 
$$
|j_{\ell,k}^2-\alpha^2 \ell^2 | \geq c_\alpha j_{\ell,k}
$$
holds for $(\ell,k) \not \in \cI_{\alpha,0}^0$.
%Moreover, ?? implies that $\cI_{\alpha,m}^0$ is finite for each $\alpha \in 
%\cA$ and $m \geq 0$.
\begin{lemma} \label{Corollary: Spectrum has no accumulation point}
	Let $\alpha \in \cA$ and $m \in \R$.
	Then $\cI_{\alpha,m}^0$ is finite and there exists $c_{m}>0$ such 
	that any $(\ell,k) \not \in \cI_{\alpha,m}^0$ satisfy
	$$
	|j_{\ell,k}^2 - \alpha^2 \ell^2 + m| \geq c_{m} j_{\ell,k} .
	$$
	In particular, the spectrum of $L_{\alpha,m}$ has no finite accumulation 
	points.
\end{lemma}
\begin{proof}
	Let $c_\alpha>0$ be given as above.
	We first note that
	\begin{align*}
		|j_{\ell,k}^2 - \alpha^2 \ell^2 + m| & = (j_{\ell,k} + \alpha \ell) 
		\left| j_{\ell,k}-\alpha \ell + \frac{m}{j_{\ell,k}+\alpha \ell} 
		\right|,
	\end{align*}
    so the fact that $\cI_{\alpha,0}^0$ is finite by assumption implies that 
    $\cI_{\alpha,m}^0$ is finite as well. 
    Moreover, there exist $\ell_0, k_0 \in \N$ such that
	$$
	|j_{\ell,k}^2 - \alpha_n^2 \ell^2 + m| \geq (j_{\ell,k} + \alpha \ell) 
	\frac{c_\alpha}{2}
	$$
	holds for all $(\ell,k) \not \in \cI_{\alpha,m}^0$ with $\ell \geq 
	\ell_0$, $k \geq k_0$. Setting
	$$
	c_{m} \coloneqq \min \left\{ \frac{c_\alpha}{2} , \min_{\substack{(\ell,k) 
	\not \in \cI_{\alpha,m}^0 \\ \ell \leq 
	\ell_0, k \leq k_0  }} \left| 
	j_{\ell,k}-\alpha \ell + \frac{m}{j_{\ell,k}+\alpha \ell} \right|   
	\right\} >0
	$$
	then completes the proof.
\end{proof}
Next, we recall the eigenfunctions $\phi_{\ell,k},\psi_{\ell,k}$ given in 
\eqref{eq:eigenfunction-representation} and set
$$
E_{\alpha,m} \coloneqq \left\{  u \in L^2(\B): \sum_{\ell=0}^\infty \sum_{k=1}^\infty {\left|j_{\ell,k}^2-\alpha^2 \ell^2 +m \right|} \left(|\langle u,\phi_{\ell,k} \rangle|^2 + |\langle u,\psi_{\ell,k} \rangle|^2 \right) < \infty  \right\} 
$$
for $\alpha \in \cA$, $m \in \R$ and endow $E_{\alpha,m}$ with the scalar product
\begin{align*}
\langle u,v \rangle_{\alpha,m} \coloneqq & \sum_{\ell=0}^\infty \sum_{k=1}^\infty {\left|j_{\ell,k}^2-\alpha^2 \ell^2 +m\right|} \left(\langle u,\phi_{\ell,k} \rangle \langle v,\phi_{\ell,k} \rangle + \langle u,\psi_{\ell,k} \rangle \langle v,\psi_{\ell,k} \rangle \right) \\
& + \sum_{(\ell,k) \in \cI_{\alpha,m}^0}  \left(\langle u,\phi_{\ell,k} \rangle \langle v,\phi_{\ell,k} \rangle + \langle u,\psi_{\ell,k} \rangle \langle v,\psi_{\ell,k} \rangle \right)  	.
\end{align*}
In the following, $\| \cdot\|_{\alpha,m}$ denotes the norm induced by $\langle 
\cdot,\cdot \rangle_{\alpha,m}$. 
\begin{remark} 
	For fixed $\alpha \in \cA$, the norm $\|\cdot\|_{\alpha,m}$ is equivalent 
	to 
	$\|\cdot\|_{\alpha,0}$ and $E_{\alpha,m}=E_{\alpha,0}$, i.e., the spaces 
	are 
	equal as sets. Nonetheless, the use of an $m$-dependent scalar product is 
	useful for the variational methods we will employ below. 
\end{remark}
We now consider the following decomposition associated to the eigenspaces of
positive, zero and negative eigenvalues of $L_{\alpha,m}$, respectively:
\begin{align*}
E_{\alpha,m}^+ & \coloneqq \left\{ u \in E_{\alpha,m}: \ \int_\B u \,
\phi_{\ell,k} \, dx = \int_\B u \, \psi_{\ell,k} \, dx = 0 \ \text{for 
$(\ell,k) 
\in \cI_{\alpha,m}^0 \cup \cI_{\alpha,m}^-$}  \right\}  \\
E_{\alpha,m}^0 & \coloneqq  \left\{ u \in E_{\alpha,m}: \ \int_\B u \,
\phi_{\ell,k} \, dx = \int_\B u \, \psi_{\ell,k} \, dx = 0 \ \text{for 
$(\ell,k) 
\in \cI_{\alpha,m}^+ \cup \cI_{\alpha,m}^-$}  \right\}  \\
E_{\alpha,m}^- & \coloneqq  \left\{ u \in E_{\alpha,m}: \ \int_\B u \,
\phi_{\ell,k} \, dx = \int_\B u \, \psi_{\ell,k} \, dx = 0 \ \text{for 
$(\ell,k) 
\in \cI_{\alpha,m}^+ \cup \cI_{\alpha,m}^0$}  \right\}  
\end{align*}
so that, in particular,
$$
E_{\alpha,m} = 	E_{\alpha,m}^+ \oplus 	E_{\alpha,m}^0 \oplus 	E_{\alpha,m}^- 
= E_{\alpha,m}^+ \oplus F_{\alpha,m} ,
$$
where we have set $F_{\alpha,m}\coloneqq E_{\alpha,m}^0 \oplus E_{\alpha,m}^-$. In the following, we will routinely write
$$
u=u^+ + u^0 + u^-
$$
where $u^+ \in 	E_{\alpha,m}^+$, $u^0 \in 	E_{\alpha,m}^0$, $u^- \in 	
E_{\alpha,m}^-$ are uniquely determined. The use of the norm 
$\|\cdot\|_{\alpha,m}$ allows us to write
$$
\langle L_{\alpha,m} u,u \rangle_{L^2(\B)} = \int_{\B }\left(|\nabla 
u|^2-\alpha^2 |\del_\theta u|^2 + m u^2\right) \, dx = 
\|u^+\|_{\alpha,m}^2-\|u^-\|_{\alpha,m}^2.
$$
Importantly, $E_{\alpha,m}$ has the following embedding properties:
\begin{proposition} \label{Proposition: K properties}
	Let $p \in (2,4)$, $\alpha \in \cA$ and $ m \in \R$. Then $E_{\alpha,m} \subset L^p(\B)$ and the embedding
	$$
	E_{\alpha,m} \hookrightarrow L^p(\B)
	$$
	is compact.
\end{proposition}
\begin{proof}
	Because of the compact embedding $H^\frac{1}{2}(\B) \hookrightarrow 
	L^p(\B)$ it is enough to show that the embedding
	$$
	E_{\alpha,m} \hookrightarrow H^\frac{1}{2}(\B)
	$$
	is well-defined and continuous. We first note that it suffices to consider 
	$u \in 	E_{\alpha,m}^+ \oplus 	E_{\alpha,m}^-$, since the space $	
	E_{\alpha,m}^0$ is finite--dimensional and only contains smooth functions. 
	By Lemma~\ref{Corollary: Spectrum 
	has no accumulation point}, there exists $c>0$ such that 
	$$
	|j_{\ell,k}^2 - \alpha^2 \ell^2 + m| \geq c j_{\ell,k} 
	$$
	holds for $(\ell,k) \not\in \cI_{\alpha,m}^0$. This implies
	\begin{align*} 
	\|u\|_{H^\frac{1}{2}}^2 &= \sum_{\ell=0}^\infty \sum_{k=1}^\infty j_{\ell,k} \left(|\langle u,\phi_{\ell,k} \rangle|^2 + |\langle u,\psi_{\ell,k} \rangle|^2 \right) \\
	& \leq \frac{1}{c} \sum_{\ell=0}^\infty \sum_{k=1}^\infty |j_{\ell,k}^2 - 
	\alpha^2 \ell^2 +m|^2 \left(|\langle u,\phi_{\ell,k} \rangle|^2 + |\langle 
	u,\psi_{\ell,k} \rangle|^2 \right) \\
	& = \frac{1}{c} \|u\|_{{\alpha,m}}^2 
	\end{align*}
	and thus the claim.
\end{proof}
\begin{remark} %\label{Remark: Optimality of Exponents}
	It is natural to ask for the optimal $q>2$ such that the preceding 
	proposition holds for $p \in (2,q)$. We conjecture that $q=10$ due to two 
	observations:
	
	Firstly, $q=10$ appears in the degenerate 
	elliptic case $\alpha=1$ treated in \cite{Kuebler-Weth} as the critical 
	exponent for Sobolev-type embeddings for the associated degenerate 
	operator. Secondly, this exponent also appears in a Poho\v{z}aev-type 
	identity in \cite{Lupo-Payne: Critical exponents} with respect to related
	semilinear problems involving the Tricomi operator.
\end{remark}
In particular, the map
$$
I_p: E_{\alpha,m} \to \R, \qquad I_p(u)\coloneqq \frac{1}{p} \int_\B |u|^p \, 
dx = \frac{1}{p} \|u\|_p^p 
$$
is well-defined and continuous for $p \in (2,4)$.
We note the following properties corresponding to the conditions of Theorem 35 in \cite{Szulkin-Weth}.
\begin{lemma}  \label{Lemma: I properties}
	Let $\alpha \in \cA$, $m \in \R$ and $p \in (2,4)$. Then the following 
	properties hold:
	\begin{enumerate}
		\item[(i)]
		$\frac{1}{2} I_p'(u) u > I_p(u)>0$ for all $u \not \equiv 0$ and $I_p$ is weakly lower semicontinuous.
		\item[(ii)]
		$I_p'(u)=o(\|u\|_{\alpha,m})$ as $u \to 0$.
		\item[(iii)]
		$\frac{I_p(su)}{s^2} \to \infty$ uniformly in $u$ on weakly compact subsets of $E_{\alpha,m} \setminus \{ 0\}$ as $s \to \infty$.
		\item[(iv)]
		$I_p'$ is a compact map.
	\end{enumerate}	
\end{lemma}
\begin{proof}
	The properties (i),(ii) and (iv)  follow from routine computations and Proposition~\ref{Proposition: K properties}, while (iii) has essentially been proved in \cite[Theorem 16]{Szulkin-Weth}, though we can give a slightly simpler argument in this case:
	
	Let $W \subset E_{\alpha,m} \setminus \{ 0\}$ be a weakly compact subset. We claim that
	there exists $c>0$ such that $\|u\|_p \geq c$ holds for $u \in W$. Indeed, if this was false, there would exist a sequence $(u_n)_n \subset W$ such that $u_n \to 0$ in $L^p(\B)$. The weak compactness of $W$ and Proposition~\ref{Proposition: K properties} would then imply $u_n \weakto 0$, contradicting the fact that $0 \notin W$. We thus have 
	$$
	\frac{I_p(su)}{s^2} = \frac{s^{p-2}}{p} \|u\|_p^p \geq  \frac{c^{p}}{p} s^{p-2} 
	$$
	and clearly the right hand side goes to infinity uniformly as $s \to 
	\infty$.
\end{proof}
In the following, we always assume that $p \in (2,4)$ is fixed and consider the 
energy functional $\Phi_{\alpha,m}: E_{\alpha,m} \to \R$ given by
\begin{align*} 
\Phi_{\alpha,m}(u) 
\coloneqq \,  & \frac{1}{2} \|u^+\|_{\alpha,m}^2 - \frac{1}{2} \|u^-\|_{\alpha,m} - I_p(u) \\
 =\,  & \frac{1}{2} \int_{\B }\left(|\nabla u|^2-\alpha^2 |\del_\theta u|^2 + m 
 u^2\right) \, dx - \frac{1}{p} \int_\B |u|^p \, dx .
\end{align*}
In particular, any critical point $u \in E_{\alpha,m}$ of $\Phi_{\alpha,m}$ 
satisfies
$$
\int_{\B} |u|^{p-2} u \, \phi \, dx = \langle u^+,\phi \rangle_{\alpha,m} - 
\langle u^-,\phi \rangle_{\alpha,m} = \int_\B u \, L_{\alpha,m} \phi \, dx
$$
and can thus be interpreted as a weak solution of \eqref{Reduced equation}.
As outlined in the introduction, we will now characterize ground states of 
$\Phi_{\alpha,m}$  %, which yield 
%least energy solutions of \eqref{Reduced equation}.
by considering the generalized Nehari manifold
$$
\cN_{\alpha,m}\coloneqq \left\{ u \in E_{\alpha,m} \setminus F_{\alpha,m}: \ 
\Phi_{\alpha,m}'(u) u=0 \ \text{and } \Phi_{\alpha,m}'(u)v = 0 \ \text{for all 
$v \in F_{\alpha,m}$} \right\} .
$$
In particular, $\cN_{\alpha,m}$ contains all nontrivial critical points of $\Phi$. 
Consequently, the value 
$$
c_{\alpha,m} \coloneqq \inf_{u \in \cN_{\alpha,m}} \Phi_{\alpha,m}(u) 
$$
is the ground state energy in the sense that any critical point $u \in 
E_{\alpha,m} \setminus \{0\}$ of $\Phi_{\alpha,m}$ satisfies 
$\Phi_{\alpha,m}(u) \geq c_{\alpha,m}$. This motivates the following definition.
\begin{definition}
	Let $\alpha \in \cA$, $m \in \R$ and $p \in (2,4)$. We call a function $u 
	\in E_{\alpha,m}$ a \textbf{ground state solution} of \eqref{Reduced 
	equation}, if $u$ is a critical point of $\Phi_{\alpha,m}$ and satisfies 
	$\Phi_{\alpha,m}(u)=c_{\alpha,m}$.
\end{definition}
In order to show that ground state solutions exist, we wish to verify 
that $\Phi_{\alpha,m}$ satisfies condition $(B_2)$ from \cite{Szulkin-Weth}. To 
this end, we let $u \in E_{\alpha,m} \setminus F_{\alpha,m}$ and consider
$$
\widehat E_{\alpha,m}(u)\coloneqq \{ t u + w: \ t \geq 0, \ w \in F_{\alpha,m} \}= \R^+ u \oplus F_{\alpha,m} .
$$
Importantly, $u \in \cN_{\alpha,m}$ if and only if $u$ is a critical point of 
$\Phi_{\alpha,m} \big|_{\widehat E_{\alpha,m}(u)}$. Moreover, we have $\widehat 
E_{\alpha,m}(u)=\widehat E_{\alpha,m}(tu^+)$ for all $t \geq 0$, $u \in 
E_{\alpha,m} \setminus F_{\alpha,m}$, hence when considering $\widehat 
E_{\alpha,m}(u)$ we may always assume $u \in E^+_{\alpha,m}$. This will be 
useful in the 
following.
\begin{lemma} \label{Lemma: Maximum along halfspaces}
	For each $u \in E_{\alpha,m} \setminus F_{\alpha,m}$ there exists a unique 
	nontrivial critical point $\hat m(u)$ of $\Phi_{\alpha,m} \big|_{\widehat 
	E_{\alpha,m}}$. Moreover, $\hat m(u)$ is the unique global maximum of 
	$\Phi_{\alpha,m} \big|_{\widehat E_{\alpha,m}}$.
\end{lemma}
\begin{proof}
	The following argument is essentially taken from \cite[Proposition 39]{Szulkin-Weth}. 
	Without loss of generality we may assume $u \in E^+_{\alpha,m}$ and 
	$\|u\|_{\alpha,m}=1$.
	
	\hypertarget{Claim 1}{\textbf{Claim 1:}} There exists $R>0$ such that 
	$\Phi_{\alpha,m}(v) \leq 0$ holds for $v \in \widehat E_{\alpha,m}$ and 
	$\|v\|_{\alpha,m} \geq R$.
	\\
	Indeed, if this was false there would exist a sequence $(v_n)_n \subset 
	\widehat E_{\alpha,m}(u)$ such that $\|v_n\|_{\alpha,m} \to \infty$ and 
	$\Phi_{\alpha,m}(v_n) >0$. 
	Setting $w_n \coloneqq \frac{v_n}{\|v_n\|_{\alpha,m}}$ we may pass to a weakly convergent subsequence and note that
	$$
	\begin{aligned}
	0 & < \frac{\Phi_{\alpha,m}(v_n)}{\|v_n\|_{\alpha,m}^2} = \frac{1}{2} 
	\|w_n^+\|_{\alpha,m}^2 - \frac{1}{2} \|w_n^-\|_{\alpha,m}^2 - \frac{1}{p} 
	\frac{\left\|\|v_n\|_{\alpha,m} w_n \right\|_p^p}{\|v_n\|_{\alpha,m}^2}  \\
	& 	\leq \|w_n\|_{\alpha,m}^2 - \frac{I(|\|v_n\|_{\alpha,m} 
	w_n)}{\|v_n\|_{\alpha,m}^2}
\end{aligned}
	$$
	so that Lemma~\ref{Lemma: I properties}(iii) implies $0 < 
	\frac{\Phi_{\alpha,m}(v_n)}{\|v_n\|_{\alpha,m}^2}  \to -\infty$ if the weak 
	limit is nonzero. Hence we must have $w_n \weakto 0$. Moreover, the 
	inequality 
	above also implies $\|w_n^+\|_{\alpha,m} \geq \|w_n^-\|_{\alpha,m}$. If 
	$w_n^+ \to 0$, the latter also implies $w_n^- \to 0$ and therefore
	$$
	\|w_n^0\|_{\alpha,m}^2=1-\|w_n^+\|_{\alpha,m}-\|w_n^-\|_{\alpha,m}^2 \to 1 .
	$$
	The fact that $E_{\alpha,m}^0$ is finite-dimensional then implies that 
	$w_n^0$ converges to a nontrivial function, which contradicts $w_n \weakto 
	0$. Hence $w_n^+$ cannot converge to zero and we may therefore pass to a 
	subsequence such that $\|w_n^+\|_{\alpha,m} \geq \gamma$ holds from some 
	$\gamma>0$ and all $n$. However, by definition of $\widehat 
	E_{\alpha,m}(u)$ we must have $w_n^+ = u \|w_n^+\|_{\alpha,m}$ and 
	therefore there exists $c>0$ such that $w_n^+ \to c u$ holds after passing 
	to a subsequence, contradicting $w_n \weakto 0$. This proves 
	\hyperlink{Claim 1}{Claim 1}.
	
	Next, we note that Lemma~\ref{Lemma: I properties} yields 
	$\Phi_{\alpha,m}(tu) = 
	\frac{t^2}{2} + o(t^2)$ as $t \to 0$ and therefore 
	$$
	\sup_{\widehat 	E_{\alpha,m}(u)} \Phi_{\alpha,m}>0.
	$$
	Now \hyperlink{Claim 1}{Claim 1} 
	implies that any maximizing sequence $(v_n)_n \subset \widehat 
	E_{\alpha,m}(u)$ must remain bounded, so we may assume $v_n \weakto v$ 
	after 
	passing to a subsequence. Moreover, recalling that
	$$
	\Phi_{\alpha,m}(v_n)=\frac{\|v_n^+\|_{\alpha,m}^2}{2} - 
	\frac{\|v_n^-\|_{\alpha,m}^2}{2}  - I_p(v_n) ,
	$$
	we can use that $v_n^+$ is a multiple of $u$, while the norm 
	$\|\cdot\|_{\alpha,m}$ and $I_p$ are weakly lower semicontinuous on 
	$E_{\alpha,m}$, making $\Phi_{\alpha,m}$ weakly upper semicontinuous on 
	$\widehat E_{\alpha,m}(u)$. It thus follows that $\sup_{\widehat 
	E_{\alpha,m}(u)} \Phi_{\alpha,m}$ is attained by a critical point $u_0$ of 
	$\Phi_{\alpha,m} \big|_{\widehat E_{\alpha,m}(u)}$. Noting that $\sup_{t 
	\geq 0} \Phi_{\alpha,m}(tu) >0$ since $u \in E_{\alpha,m}^+$, it follows 
	that $u_0 \in \cN_{\alpha,m}$.
	
	It remains to prove that this is the only critical point of 
	$\Phi_{\alpha,m} \big|_{\widehat E_{\alpha,m}(u)}$.
	To this end, we let $w \in E_{\alpha,m}$ such that $u_0+w \in \widehat E_{\alpha,m}(u)$. Since $\widehat E_{\alpha,m}(u)=\widehat E_{\alpha,m}(u_0)$, there exists $s \geq -1$ such that $u_0+w=(1+s) u_0 +v$ for some $v \in F_{\alpha,m}$. Setting
	$$
	\begin{aligned} 
	B(v_1,v_2) & \coloneqq \int_\B \left( \nabla v_1 \cdot \nabla v_2 - 
	\alpha^2 
	(\del_\theta v_1) (\del_\theta v_2)  + m v_1 v_2\right) \, dx \\
	& = \langle 
	v_1^+, v_2^+ \rangle_{\alpha,m} -  \langle v_1^-, v_2^- \rangle_{\alpha,m}
	\end{aligned}
	$$
	we then have
	$$
	\begin{aligned}
		\Phi_{\alpha,m}(u_0+w) - \Phi_{\alpha,m}(u_0) & =  \frac{1}{2} \left( 
		B((1+s) u_0 +v,(1+s) u_0 +v) - B(u_0,u_0) \right) \\
		& \quad - I_p((1+s) u_0 +v) + I_p(u_0) \\
		& =  - \frac{\|v^-\|_{\alpha,m}^2}{2} + B \left(u_0, s 
		\left(\frac{s}{2}-1\right)u_0+(1+s)v\right) \\
		& \quad  - I_p((1+s) u_0 +v) + I_p(u_0) ,
	\end{aligned}
	$$
	where the fact that $\Phi_{\alpha,m}'(u_0)(\cdot) = B(u_0, \cdot) -  
	I_p'(u_0)(\cdot)=0$ then implies
	$$
	\begin{aligned}
		   &  B \left(u_0, s \left(\frac{s}{2}-1\right)u_0+(1+s)v\right)  - I_p((1+s) u_0 +v) + I_p(u_0) \\
		 = &  I_p'(u_0) \left(s \left(\frac{s}{2}-1\right)u_0+(1+s)v\right) - I_p((1+s) u_0 +v) + I_p(u_0) \\
		 = &  \int_\B \left( |u_0|^{p-2} u_0 \left(s 
		 \left(\frac{s}{2}-1\right)u_0+(1+s)v\right) - \frac{1}{p} |(1+s) u_0 
		 +v|^p + \frac{1}{p} |u_0|^p \right) \, dx \\
		 < & 0
	\end{aligned}
	$$
	by \cite[Lemma 2.2]{Szulkin-Weth: Paper}.
\end{proof}
We can then give the following existence result.
\begin{proposition} \label{Prop: Existence of Ground States and Minimax 
Characterization}
	Let $\alpha \in \cA$, $m \in \R$ and $p \in (2,4)$. Then $c_{\alpha,m}$ is 
	positive and attained by a critical point of $\Phi_{\alpha,m}$. In 
	particular, \eqref{Reduced equation} thus has a ground state solution.
	Moreover, 
	$$
	c_{\alpha,m} = \inf_{w \in E_{\alpha,m} \setminus F_{\alpha,m}} \,  \max_{w 
	\in \widehat E_{\alpha,m}(u)} \Phi_{\alpha,m}(w) 
	$$
	holds.
\end{proposition}
\begin{proof}
	Note that Lemma~\ref{Lemma: I properties} and Lemma~\ref{Lemma: Maximum 
	along halfspaces} imply that $\Phi_{\alpha,m}$ satisfies the conditions of 
	\cite[Theorem 35]{Szulkin-Weth}.
\end{proof}
In particular, this implies Theorem~\ref{Theorem: Existence of nonradial 
solutions - Introduction}(i).
Notably, this minimax characterization of $c_{\alpha,m}$ will allow us to 
compare the ground 
state energy to the minimal energy among radial solutions, which we estimate in 
the following.
\begin{lemma}  \label{Lemma: Lower bound for radial energies}
	Let $p > 2$ and $m \geq 0$, where $\lambda_1>0$ denotes the first 
	Dirichlet eigenvalue of $-\Delta$ on $\B$. Then there exists a unique 
	positive 
	radial solution $u_m \in H^1_{0,rad}(\B)$ of \eqref{Reduced 
	equation}, 
	i.e., satisfying 
	\begin{equation*}
	\left\{ 
	\begin{aligned} 
	-\Delta u  +m u & = |u|^{p-2} u \quad && \text{in $\B$} \\
	u & = 0 && \text{on $\del \B$.}
	\end{aligned}
	\right.
	\end{equation*}
	Moreover, there exists $c>0$ such that
	$$
	\beta_m^{rad} \coloneqq \Phi_{\alpha,m}(u_{m}) \geq c m^\frac{2}{p-2}
	$$
	holds for all $\alpha>1$ and $m \geq 0$. 
	%In particular, we have 
	%$$
	%\inf_{m \geq 0} \Phi_{\alpha,m}(u_m) >0 .
	%$$
\end{lemma}
\begin{proof}
	We consider the functional 
	\begin{align*}
	J_{m} & : H^1_{0,rad}(\B) \to \R \\
	J_{m}(u) & \coloneqq \frac{1}{2} \int_{\B }\left(|\nabla u|^2-\alpha^2 
	|\del_\theta u|^2 + m u^2\right) \, dx - \frac{1}{p} \int_\B |u|^p \, dx
	\end{align*}
	which satisfies $J_m(u)=\Phi_{\alpha,m}(u)$ for every $u \in 
	H^1_{0,rad}(\B)$ and 
	$\alpha>1$.
	For $m \geq 0$ we consider the classical Nehari manifold
	$$
	\cN_m^{rad} \coloneqq \left\{ u \in H^1_{0,rad}(\B) \setminus \{0 \}: \ 
	J_{m}'(u) u = 0  \right\} .
	$$
	Clearly, any nontrivial radial critical point $u$ of $\Phi_{\alpha,m}$ is 
	contained in $\cN_m^{rad}$. 	
	Moreover, the map
	$$
	(0,\infty) \to \R, \quad t \mapsto J_m(t u)
	$$
	attains a unique maximum $t_u>0$ for each $u \in  H^1_{0,rad}(\B) \setminus 
	\{ 0\}$ and simple computations yield
	$$
	J_m(t_u u) = 
	\sup_{t \geq 0} J_m(t u)
	= \left(\frac{1}{2}-\frac{1}{p}\right) \left( \frac{\int_\B \left( |\nabla 
	u|^2+mu^2 \right) \, dx }{\left(\int_\B |u|^p \, dx \right)^\frac{2}{p}} 
	\right)^\frac{p}{p-2} 
	$$
	and $t_u$ is the unique value $t>0$ such that $t u \in \cN_m$.
	It can be shown that
	$$
	\beta_{m}^{rad} \coloneqq \inf_{u \in \cN_m^{rad}} J_m (u)
	$$
	is a critical value of $J_m$, see e.g. \cite{Szulkin-Weth}. Moreover, the 
	principle of symmetric criticality (see e.g. \cite{Palais}) shows that 
	$\beta_{m}^{rad}$ is in fact a critical value of $\Phi_{\alpha,m}$ and 
	attained by a unique positive radial function $u_m$. This proves the first 
	part of the theorem.
	
	Next, we note that the characterization above gives
	\begin{equation} \label{eq: radial energy characterization}
	\begin{aligned} 
	\beta_{m}^{rad} & =
	\inf_{ u \in  H^1_{0,rad}(\B)\setminus \{0\} } \sup_{t \geq 0} J_m(t u)
	\\
	&= \inf_{ u \in  H^1_{0,rad}(\B) \setminus \{0\} 
	}\left(\frac{1}{2}-\frac{1}{p}\right) \left( \frac{\int_\B \left( |\nabla 
	u|^2+mu^2 \right) \, dx }{\left(\int_\B |u|^p \, dx \right)^\frac{2}{p}} 
	\right)^\frac{p}{p-2}  .
	\end{aligned}
	\end{equation}
	In the following, we assume $m > 0$ and let $B_{\sqrt{m}}$ denote the ball 
	of radius 
	$\sqrt{m}$ centered at the origin. We then consider the function 
	$v_m \in H^1_0(B_{\sqrt{m}})$ given by
	$$
	v_m(x)=m^{-\frac{1}{p-2}} u_m\left(\frac{x}{\sqrt{m}}\right).
	$$
	%where, here and in the following, $B_{\sqrt{m}}$ denotes the ball of 
	%radius 	$\sqrt{m}$ centered at the origin.
	Then
	\begin{align*}
	\frac{\int_\B \left( |\nabla 
		u_m|^2+ m u_m^2 \right) \, dx }{\left(\int_\B 
		|u_m^2|^p \, dx \right)^\frac{2}{p}} 
	&= m^{\frac{2}{p}} 
	\frac{\int_{B_{\sqrt{m}}} \left( |\nabla v_m|^2  + v_m^2 \, 
		\right)dx}{\left(\int_{B_{\sqrt{m}} } |v_m|^p \, dx 
		\right)^\frac{2}{p}}  
	\\
	& \geq m^{\frac{2}{p}}  \inf_{v \in H^1({\R^N}) \setminus \{0\}} 
	\frac{\int_{\R^N} \left( 
		|\nabla v|^2  + v^2 \, \right)dx}{\left(\int_{\R^N } |v|^p \, dx 
		\right)^\frac{2}{p}} .
	\end{align*}
	Setting
	$$
	C_p \coloneqq  \inf_{v \in H^1({\R^N}) \setminus \{0\}} 
	\frac{\int_{\R^N} \left( 
		|\nabla v|^2  + v^2 \, \right)dx}{\left(\int_{\R^N } |v|^p \, dx 
		\right)^\frac{2}{p}} >0
	$$
	we thus have
	$$
	\frac{\int_\B \left( |\nabla 
		u_m|^2+ m u_m^2 \right) \, dx }{\left(\int_\B 
		|u_m|^p \, dx \right)^\frac{2}{p}}   \geq  	C_p m^{\frac{2}{p}}  
		.
	$$
	Therefore \eqref{eq: radial energy characterization} implies
	$$
	\beta_{m}^{rad} \geq \left(\frac{1}{2}-\frac{1}{p}\right) 
	\left(C_p m ^\frac{2}{p}\right)^\frac{p}{p-2} 
	$$
	and hence the claim.
\end{proof}
We will compare the previous estimate for the radial energy with suitable 
estimates for $c_{\alpha,m}$, starting with the following result.
\begin{lemma} \label{Lemma: Existence of m with nonradial ground states}
	Let $p \in (2,4)$ and $\alpha \in \cA$. Then 
	$$
	c_{\alpha,m} \leq \left(\frac{1}{2}-\frac{1}{p}\right) |\B| \inf_{(\ell,k) 
		\in \cI_{\alpha,m}^+}\left( 
	j_{\ell,k}^2-\alpha^2 \ell^2 + m \right)^\frac{p}{p-2}
	$$
	holds for $m \in \R$.
\end{lemma}
\begin{proof}
	By Lemma~\ref{Corollary: Spectrum has no accumulation point}, there exist 
	$\ell_0,k_0 \in \N$ such that 
	$$
	\left(j_{\ell_0,k_0}^2-\alpha^2 \ell_0^2 + m \right)=\inf_{(\ell,k) 
		\in \cI_{\alpha,m}^+}\left( 
	j_{\ell,k}^2-\alpha^2 \ell^2 + m \right)
	$$
	and we set
	$$
	u_0 \coloneqq \phi_{\ell_0,k_0} \in E_{\alpha,m}^+ .
	$$
	For any $t \geq 0$ and $v \in F_{\alpha,m}$ it then holds that $\int_\B u_0 
	v \, dx=0$ and therefore
	$$
	\begin{aligned} 
	\|t u_0 + v\|_p^p & \geq |\B|^{1-\frac{p}{2}} \|t u_0 + v\|_2^p =  
	|\B|^{1-\frac{p}{2}} \left( \|t u_0\|_2^2 + \|v\|_2^2 \right)^\frac{p}{2} 
	\\
	& \geq t^p |\B|^{1-\frac{p}{2}} \|u_0\|_2^p=t^p |\B|^{1-\frac{p}{2}} .
	\end{aligned}
	$$
	This yields
	\begin{align*}
	\Phi_{\alpha,m}(t u_0 + v) & \leq \frac{t^2}{2} 
	\left(j_{\ell_0,k_0}^2-\alpha^2 \ell_0^2 + m \right) - \frac{1}{p} \|t u_0 
	+ v\|_p^p \\
	& \leq \frac{t^2}{2} \left(j_{\ell_0,k_0}^2-\alpha^2 \ell_0^2 + m \right) - 
	\frac{t^p}{p} |\B|^{1-\frac{p}{2}} .
	\end{align*}
	A straightforward computation shows that the right hand side attains a 
	unique global 
	maximum 
	in
	$$
	t^*=\left(j_{\ell_0,k_0}^2-\alpha^2 \ell_0^2 + m \right)^\frac{1}{p-2} 
	|\B|^{\frac{1}{2}}
	$$
	and therefore
	$$
	\Phi_{\alpha,m}(t u_0 + v) \leq \left(\frac{1}{2}-\frac{1}{p}\right) |\B| 
\left(j_{\ell_0,k_0}^2-\alpha^2 \ell_0^2 + m \right)^\frac{p}{p-2} .
	$$
	In particular, this gives
	$$
	\max_{w 
		\in \widehat E_{\alpha,m}(u_0)} \Phi_{\alpha,m}(w) \leq  
		\left(\frac{1}{2}-\frac{1}{p}\right) |\B| 
		\left(j_{\ell_0,k_0}^2-\alpha^2 \ell_0^2 + m \right)^\frac{p}{p-2}
	$$
	and Proposition~\ref{Prop: Existence of Ground States and Minimax 
	Characterization} then finally implies
	$$
	c_{\alpha,m} = \inf_{w \in E_{\alpha,m} \setminus F_{\alpha,m}} \,  \max_{w 
	\in \widehat E_{\alpha,m}(u)} \Phi_{\alpha,m}(w) \leq  
	\left(\frac{1}{2}-\frac{1}{p}\right) |\B| \left(j_{\ell_0,k_0}^2-\alpha^2 
	\ell_0^2 + m \right)^\frac{p}{p-2}
	$$
	as claimed.
\end{proof}
The previous results allow us to deduce the existence of nonradial ground 
states whenever 
$$
\left(\frac{1}{2}-\frac{1}{p}\right) |\B| \inf_{(\ell,k) 
	\in \cI_{\alpha,m}^+}\left( 
j_{\ell,k}^2-\alpha^2 \ell^2 + m \right)^\frac{p}{p-2}<\beta_m^{rad} 
$$
holds. To this end, we estimate the growth of the left hand side as $m \to 
\infty$. % and compare this to Lemma~\ref{Lemma: Lower bound for radial 
%energies}.
\begin{proposition} \label{Proposition: Eigenvalue eps bound}
	Let $\alpha \in \cA$.
	Then there 
	exist constants $C>0$, $m_0>0$ such that
	$$
	\inf_{(\ell,k) 
		\in \cI_{\alpha,m}^+}\left( 
	j_{\ell,k}^2-\alpha^2 \ell^2 + m \right) \leq C m^\frac{1}{2}
	$$
	holds for $m > m_0$.
\end{proposition}
\begin{proof}
	By Proposition~\ref{Prop: Properties of Bessel function zeros} 
	we have
		\begin{equation} \label{eq:jl1-estimate}
			\ell + \frac{|a_1|}{2^\frac{1}{3}} \ell^\frac{1}{3} < j_{\ell,1} < 
			l + 
			\frac{|a_1|}{2^\frac{1}{3}} \ell^\frac{1}{3} + \frac{3}{20} |a_1|^2 
			\frac{2^\frac{1}{3}}{\ell^\frac{1}{3}} ,
		\end{equation}
		where $a_1$ denotes the first negative zero of the Airy 
		function $\mathrm{Ai}(x)$. In particular, this implies that there 
		exists $\ell_0 \in \N$ such that the map 
		$$
		\ell \mapsto j_{\ell,1}^2 - \alpha^2 \ell 
		$$
		is strictly decreasing for $\ell \geq \ell_0$. 
		Taking $m_0 > \alpha^2 \ell_0^2 - j_{\ell_0,1}^2$ we thus find that for 
		any $m>m_0$ there exists $\ell \geq \ell_0$ such that
		$$
		m \in \left(\alpha^2 \ell^2 - j_{\ell,1}^2 ,\alpha^2 (\ell+1)^2 - 
		j_{\ell+1,1}^2  \right].
		$$
		In the following, we fix such $m$ and $\ell$ and note that since 
		$j_{\ell,1} < j_{\ell+1,1}$, we have
		$$
		\begin{aligned}
			0<j_{\ell,1}^2 - \alpha \ell^2 - (j_{\ell+1,1}^2-\alpha(\ell+1)^2)
			& =j_{\ell,1}^2 - j_{\ell+1,1}^2  + \alpha^2 
			\left((\ell+1)^2-\ell^2\right) \\
			& \leq 	2\alpha^2 \ell +\alpha^2 
		\end{aligned}
		$$
		for $\ell \geq \ell_0$, and therefore 
		$$
		0 < j_{\ell,1}^2 - \alpha \ell^2 + m \leq 2\alpha^2 \ell 
		+\alpha^2  .
		$$
%		for any $m \in (\alpha^2 \ell^2 - j_{\ell,1}^2 ,\alpha^2 	(\ell+1)^2 
%- 		j_{\ell+1,1}^2  ]$.
		Importantly, \eqref{eq:jl1-estimate} implies that there exists 
		$C=C(\alpha)>0$ independent of $m$ such that
		$$
		 2\alpha^2 \ell 
		+\alpha^2 \leq C \left( \alpha^2 \ell^2 - j_{\ell,1}^2 
		\right)^\frac{1}{2}
		$$
		holds for $\ell \geq \ell_0$, after possibly enlarging $\ell_0$.
		Ultimately, we thus find that 
		$$
		0 < j_{\ell,1}^2 - \alpha \ell^2 + m \leq C \left( \alpha^2 \ell^2 - 
		j_{\ell,1}^2 
		\right)^\frac{1}{2}  \leq C m^\frac{1}{2} 
		$$
		holds.
		Since $C$ was independent of $m$, this completes the proof.
		% for $m \in (\alpha^2 \ell^2 - j_{\ell,1}^2 ,\alpha^2 (\ell+1)^2 
		%- 		j_{\ell+1,1}^2 ]$.
\end{proof}
%
%%%%%%%%%%%%%%%%%%%%%%%%%%%%%%%%%%%%%%%%%%%%%%%%%%%%%
%
Theorem~\ref{Theorem: Existence of nonradial solutions - Introduction}(ii) is 
now a 
direct consequence of the following more general result.
\begin{theorem}
	Let $\alpha \in \cA$ and $p \in (2,4)$ be fixed. Then there exists $m_0>0$ 
	such that the ground states of \eqref{Reduced equation} are nonradial for  
	$m >m_0$.
\end{theorem}
\begin{proof}
%	Let $\eps>0$. 
	Lemma~\ref{Lemma: Existence of m with nonradial ground states} and  
	Proposition~\ref{Proposition: Eigenvalue eps bound} imply that there exist 
	$C>0$, $m_0>0$ such that
	$$
	c_{\alpha,m} \leq \left(\frac{1}{2}-\frac{1}{p}\right) |\B| C 
	m^\frac{{p}}{2(p-2)}
	$$
	holds for $m>m_0$. On the other hand, Lemma~\ref{Lemma: Lower bound for 
	radial 
	energies} gives
	$$
	\beta_m^{rad} \geq c m^\frac{2}{p-2}
	$$
	with a constant $c>0$ independent of $m$. Noting that the assumption $p<4$ 
	implies $\frac{{p}}{2(p-2)}< \frac{2}{p-2}$, it 
	follows that
	$$
	c_{\alpha,m} < \beta_m^{rad}
	$$
	holds for $m>m_0$, after possibly enlarging $m_0$.
\end{proof}
%
%
%%
%%%%%%%%%%%%%%%%%%%%%%%%%%%%%%%%%%%%%%%%%%%%%%%%%%%%%%%%%%%%%%%%%%%%%%%%%%%%%%%%%%%%%%%%%%%%%%%%%%%%%%%%%%%%%
%
%
%
\appendix
\section{Complex-valued Solutions}   \label{Appendix: Complex Case}
Throughout this section we assume that all functions are complex-valued and 
that $p>2$ is fixed. In this case, the eigenspaces
$$
V_k \coloneqq \left\{ u \in H^1_0(\B): \del_\theta u = i k u  \right\} 
$$
are nonempty for $k \in \N$. This observation can be used to find 
complex-valued solutions of \eqref{Reduced equation} as stated in the following.
\begin{theorem} %\label{Theorem: Existence in the complex case}
	Let $\alpha>1$, $m>0$ and $k \in \N$ be chosen such that 
	\begin{equation} \label{eq:Complex Eigenvalue Condition}
	m-\alpha^2 k^2>-\lambda_1 ,
	\end{equation}
	where $\lambda_1>0$ denotes the first Dirichlet eigenvalue of $-\Delta$ on $\B$.
	Then there exists a weak solution $u \in V_k$ of \eqref{Reduced equation}. In particular, this solution is nonradial.
\end{theorem}
We point out that the solutions found in the preceding theorem cannot be 
real-valued and are thus distinct from the solutions found in 
Theorem~\ref{Theorem: Existence of nonradial solutions - Introduction}.

\begin{proof}%[Proof of Theorem~\ref{Theorem: Existence in the complex case}]
Inspired by \cite{Taylor}, the proof is based on a constrained minimization argument for the functional
$$
\begin{aligned}
	J_{\alpha,m}& : H^1_0(\B) \to \R, \\
	J_{\alpha,m}(u) & \coloneqq\frac{1}{2} \int |\nabla u|^2 -\alpha^2 
	|\del_\theta u|^2  + m u^2 \, dx .
\end{aligned} 
$$
Importantly, for $u \in V_k$ we have
$$
J_{\alpha,m}(u) =\frac{1}{2} \int |\nabla u|^2  + (m-\alpha^2 k^2) u^2 \, dx 
$$
and our goal is to minimize $J_{\alpha,m}$ on $V_k$ subject to the constraint
$$
I(u)\coloneqq\|u\|_p^p=1 .
$$
To this end, we let $(u_n)_n \subset V_k$ be a constrained minimizing sequence, 
i.e., $I(u_n)=1$ for all $n$ and 
	$$
	\lim_{n \to \infty} J_{\alpha,m}(u_n) = \min_{\substack{u \in V_k \\ I(u)=1}} J_{\alpha,m}(u) .
	$$ 
	Note that $V_k$ is a closed subspace of $H^1_0(\B)$ and, by assumption, there exist $c,C>0$ such that
	$$
	c\|u\|_{H_0^1(\B)}^2 \leq J_{\alpha,m}(u) \leq C \|u\|_{H^1(\B)}^2 
	$$
	holds for $u \in V_k$, which implies that the sequence $(u_n)_n$ remains 
	bounded in $H^1_0(\B)$ and we may pass to a weakly convergent subsequence 
	with a weak limit $u_0 \in V_k$. The compact embedding $H^1_0(\B) 
	\hookrightarrow L^p(\B)$ then implies $I(u_0)=1$ whereas weak lower 
	semicontinuity yields $J_{\alpha,m}(u_0) \leq \liminf J_{\alpha,m}(u_n)$, 
	i.e., $u_0$ is a minimizer of $J_{\alpha,m}$ subject to the constraint 
	$I(u_0)=1$.
	
	The minimization property then implies that there exists a Lagrange multiplier $K_0 \in \R$ such that
\begin{equation} \label{eq: Weak solution condition} 
	\int \nabla u_0 \cdot \nabla \phi + (m-\alpha^2 k^2) u_0 \phi \, dx = K_0 
	\int |u_0|^{p-2} u_0 \phi \, dx
\end{equation}
holds for $\phi \in V_k$. Taking $\phi=u_0$, the condition 
\eqref{eq:Complex Eigenvalue Condition} then implies that $K_0$ must be 
positive.
We now set 
$$
E: H^1_0(\B) \to \R, \qquad E(u)\coloneqq J_{\alpha,m}(u)-K_0 I(u)  ,
$$
so that, in particular, $u_0$ is a nontrivial critical point of $E \big|_{V_k}$.

	For $t \in \R$ we then consider the action
	$$
	g_t: H^1_0(\B) \to H^1_0(\B) , \qquad 	[g_t u](x)= e^{-ik t} u(R_t (x)) ,
	$$
	where $R_t$ was defined in \eqref{eq: Rotation definition}.
	Note that $g_t$ is an isometry on $H^1_0(\B)$ and $L^p(\B)$ so that $E$ is 
	invariant with respect to $g_t$. Moreover, this defines a group action on 
	$H^1_0(\B)$ and we have
	$$
	V_k=\{u \in H^1_0(\B): g_t u=u \} .
	$$
	The principle of symmetric criticality (see e.g. \cite{Palais}) then 
	implies that $u_0$ is also a critical point of $E$ on $H^1_0(\B)$ or, 
	equivalently, \eqref{eq: Weak solution condition} holds for all $\phi \in 
	H^1_0(\B)$. But this means that $K_0^\frac{1}{p-2} u_0$ is a weak solution 
	of \eqref{Reduced equation}.
\end{proof}
By construction, the solutions found above are contained in the eigenspaces of 
the operator $\del_\theta$, i.e., for any such solution $u$ there exists $k \in 
\N$ such that $u \in V_k$ and therefore $\del_\theta u=ik u$. However, this 
implies that $|u|$ is radial.

In the following, we briefly sketch how our methods can be used to find 
complex-valued solutions $u$ of \eqref{NLKG} (which are not real-valued) such 
that the modulus $|u|$ is also nonradial. To this end, we combine the ansatz 
\eqref{eq: Rotating solution ansatz} for rotating solutions with a standing 
wave ansatz, i.e.,
$$
v(t,x) = e^{i\mu t} u(R_t (x))
$$
with $R_t$ given by \eqref{eq: Rotation definition} and $\mu >0$. This reduces 
\eqref{NLKG} to the modified problem
\begin{equation} \label{eq: Modified reduced equation}
	\left\{ 
	\begin{aligned} 
		-\Delta u + \alpha^2 \del_\theta^2 u + 2i\mu \del_\theta u +(m-\mu^2) u & = |u|^{p-2} u \quad && \text{in $\B$} \\
		u & = 0 && \text{on $\del \B$}.
	\end{aligned}
	\right.
\end{equation}
Here, the eigenvalues of the operator 
$$
L_{\alpha,m,\mu} u \coloneqq -\Delta u + \alpha^2 \del_\theta^2 u + 2i\mu \del_\theta u +(m-\mu^2) u
$$
are given by
$$
j_{\ell,k}^2 - \alpha^2 \ell^2 \pm 2 \mu \ell + (m-\mu^2)
$$
and the associated eigenfunctions are given by
$$
\phi_{\ell,k}^\pm(r,\theta) \coloneqq e^{\pm i\ell \theta} J_\ell(j_{\ell,k} r) 
, \qquad \ell \in \N_0, k \in \N .
$$
This readily implies the following analogue to Lemma~\ref{Corollary: 
	Spectrum has no accumulation point}:
\begin{lemma}
	Let the sequence $(\alpha_n)_n  \subset (1,\infty)$ be given by 
	Theorem~\ref{Theorem: Spectrum has no accumulation point}. Then for any $n 
	\in \N$ and $m \geq 0$ there exist $c_{n,m}$, $\mu_n>0$ with the following 
	property: 
	
	If $|\mu| \leq \mu_n$ and $\ell,k$ are such that $j_{\ell,k}^2 - 
	\alpha^2 \ell^2 -2 \mu \ell + (m-\mu^2) \neq 0$ holds, we have
	$$
	|j_{\ell,k}^2 - \alpha^2 \ell^2 \pm 2 \mu \ell + (m-\mu^2)| \geq c_{n,m} 
	j_{\ell,k} .
	$$
\end{lemma}
\begin{proof}
Note that 
$$
j_{\ell,k}^2 - \alpha^2 \ell^2 \pm 2 \mu \ell =(j_{\ell,k}  + 
\alpha \ell) 
\left(j_{\ell,k}  - \alpha \ell \pm  \frac{ 2 \mu \ell}{j_{\ell,k}+ \alpha 
\ell}\right) 
$$
and for $\alpha=\alpha_n$ Theorem~\ref{Theorem: Spectrum has no accumulation 
point} then implies
$$
\left|j_{\ell,k}  - \alpha_n \ell \pm  \frac{ 2 \mu \ell}{j_{\ell,k}+ \alpha_n 
	\ell} \right| \geq c_n - \mu \frac{2l}{j_{\ell,k}+\alpha_n \ell} \geq c_n - 
	\frac{2\mu}{1+\alpha_n} 
$$
for sufficiently large $\ell,k$. Setting 
$$
\mu_n \coloneqq \frac{1+\alpha_n}{2} c_n,
$$
we thus find that 
$$
\lim_{N \to \infty} \inf_{\ell,k \geq N} \left|j_{\ell,k}  - \alpha_n \ell \pm  
\frac{ 2 \mu 
\ell}{j_{\ell,k}+ \alpha_n 
	\ell} \right| > 0
$$
for $\mu< \mu_n$.  
\end{proof}
Repeating the arguments of Section~\ref{Section: Variational Formulation} 
ultimately gives the following result:
\begin{theorem}
	Let $p \in (2,4)$. Then there exists a sequence $(\alpha_n)_n  \subset (1,\infty)$ with the following properties:
	\begin{itemize} 
		\item[(i)]
		For each $n \in \N$ the problem \eqref{eq: Modified reduced equation} has a ground state solution.
		\item[(ii)]
		For each $n \in \N$ there exists $m_n>0$ such that any ground state $u$ 
		\eqref{eq: Modified reduced equation} with $\alpha = \alpha_n$ and 
		$m>m_n$ has a nonradial modulus, i.e., $|u|$ is nonradial.
		%Moreover, the set $\cM_n$ is unbounded and contains infinitely many 
		%disjoint intervals. 
	\end{itemize}
\end{theorem}

%
%
%
%
%%%%%%%%%%%%%%%%%%%%%%%%%%%%%%%%%%%%%%%%%%%%%%%%%%%%%%%%%%%%%%%%%%%%%%%%%%%%%%%%%%%%%%%%%%%%%%%%%%%%%%%%%%%%%%%%%
%%%%%%%%%%%%%%%%%%%%%%%%%%%%%%%%%%%%%%%%%%%%%%%%%%%%%%%%%%%%%%%%%%%%%%%%%%%%%%%%%%%%%%%%%%%%%%%%%%%%%%%%%%%%%%%%%
%%%%%%%%%%%%%%%%%%%%%%%%%%%%%%%%%%%%%%%%%%%%%%%%%%%%%%%%%%%%%%%%%%%%%%%%%%%%%%%%%%%%%%%%%%%%%%%%%%%%%%%%%%%%%%%%%
%

\end{document}